\documentclass[11pt]{amsart}

\usepackage{amsrefs}
\usepackage[all]{xy}
\usepackage{syntonly}
\usepackage{hyperref}
\usepackage{amsfonts}
\usepackage{amssymb}
\usepackage{amsmath}
\usepackage{amsthm}
\usepackage[inline]{enumitem}
\usepackage{tikz}
\usepackage{graphics}
\usepackage{graphicx}
\usepackage{color}
\usepackage{comment}
\usepackage{caption,lipsum}

\usepackage{lineno}

\usepackage{multirow}

\newtheorem{theorem}{Theorem}[section]
\newtheorem{corollary}[theorem]{Corollary}

\newtheorem{definition}[theorem]{Definition}
\newtheorem{lemma}[theorem]{Lemma}

\newtheorem{question}[theorem]{Question}

\newtheorem{fact}[theorem]{Fact}

\newtheorem{remark}[theorem]{Remark}

\specialcomment{com}{ \color{red} \textbf{COMMENT:} }{  \color{black}} 
\specialcomment{oldcom}{ \color{brown} }{\color{black}} 

\excludecomment{oldcom} 

\usepackage{stmaryrd}

\begin{document}
\title[]{Filtration Games and Potentially Projective Modules}

\author{Sean Cox}
\email{scox9@vcu.edu}
\address{
Department of Mathematics and Applied Mathematics \\
Virginia Commonwealth University \\
1015 Floyd Avenue \\
Richmond, Virginia 23284, USA 
}

\thanks{The author gratefully acknowledges support from Simons Foundation grant 318467, and helpful comments of the anonymous referee.}


\begin{abstract}
The notion of a \textbf{$\boldsymbol{\mathcal{C}}$-filtered} object, where $\mathcal{C}$ is some (typically small) collection of objects in a Grothendieck category, has become ubiquitous since the solution of the Flat Cover Conjecture around the year 2000.  We introduce the \textbf{$\boldsymbol{\mathcal{C}}$-Filtration Game of length $\boldsymbol{\omega_1}$} on a module, paying particular attention to the case where $\mathcal{C}$ is the collection of all countably-generated projective modules.  We prove that Martin's Maximum implies the determinacy of many $\mathcal{C}$-Filtration Games of length $\omega_1$, which in turn imply the determinacy of certain Ehrenfeucht-Fra\"iss\'{e} games of length $\omega_1$; this allows a significant strengthening of a theorem of Mekler-Shelah-Vaananen~\cite{MR1191613}.  Also, Martin's Maximum implies that if $R$ is a countable hereditary ring, the class of \textbf{$\boldsymbol{\sigma}$-closed potentially projective modules}---i.e., those modules that are projective in some $\sigma$-closed forcing extension of the universe---is closed under $<\aleph_2$-directed limits.  We also give an example of a (ZFC-definable) class of abelian groups that, under the ordinary subgroup relation, constitutes an Abstract Elementary Class (AEC) with L\"owenheim-Skolem number $\aleph_1$ in some models in set theory, but fails to be an AEC in other models of set theory.  

%This answers a question of Baldwin-Eklof-Trlifaj~\cite{MR2364195}.
\end{abstract}

\maketitle

%\tableofcontents

\section{Introduction}\label{sec_Intro}

The classic Gale-Stewart Theorem ensures that 2-player open games (of length $\omega$) are always determined.  A frequently used example of such a game is the Ehrenfeucht-Fra\"iss\'{e} game of length $\omega$ on a pair of first-order structures, which is an open game for \emph{Spoiler} (using the terminology of Spencer~\cite{MR1847951}).  It is natural to wonder to what extent does the Gale-Stewart Theorem extend to open games of length $\omega_1$.  In this context, ``open for Player X" means that every win by Player X is known at some countable stage of play.\footnote{More precisely: let $\mathcal{W}_X$ denote the set of maximal branches through the game tree that represent wins for Player X.  The game (of length $\omega_1$) is open for Player X if $\mathcal{W}_X$ is open with respect to the topology generated by sets of the form $B_t:=$ ``the maximal branches passing through $t$" where $t$ is a node in the game tree of countable height.}  The $\omega_1$-length analogue of Gale-Stewart Theorem fails, however, even when restricted to Ehrenfeucht-Fra\"iss\'{e} games.  For example, Mekler-Shelah-Vaananen~\cite{MR1191613} provide a ZFC example of an Ehrenfeucht-Fra\"iss\'{e} game of length $\omega_1$ that is not determined.  They did, however, show:
\begin{theorem}[Mekler et al.~\cite{MR1191613}, Theorem 10]\label{thm_MeklerEtAl}
Assume
\begin{equation}\label{eq_LSTOmega_2Laa}
\text{Every structure has an L(aa)-elementary substructure of size } < \aleph_2, \tag{*}
\end{equation}
where L(aa) is Shelah's \emph{Stationary Logic}.  Then whenever $G$ is a group and $F$ is a free abelian group, the Ehrenfeucht-Fra\"iss\'{e} game of length $\omega_1$ on the pair $(G,F)$ is determined.
\end{theorem}
\noindent The assumption \eqref{eq_LSTOmega_2Laa} is consistent relative to the consistency of a supercompact cardinal; Ben-David~\cite{MR506381} shows that it holds after the countably-closed Levy collapse to turn a supercompact cardinal into $\aleph_2$.\footnote{The proof of Theorem 10 in \cite{MR1191613} seems to actually use the stronger assumption that \eqref{eq_LSTOmega_2Laa} holds for the \emph{infinitary} version $L_{\omega_1,\omega_1}(aa)$ of Stationary Logic.  The version of \eqref{eq_LSTOmega_2Laa} for $L_{\omega_1,\omega_1}(aa)$ does hold after Levy collapsing to turn a supercompact cardinal into $\aleph_2$ (as in \cite{MR506381}), but implies the Continuum Hypothesis.  The infinitary version seems to be needed to express freeness in the language of groups, as in the sentence labeled (+) before Proposition 5 of \cite{MR1191613}.  The use of infinitary logic can be circumvented by augmenting the structure with enough set theory, as in \eqref{eq_H_theta_aa_aa} and \eqref{eq_H_theta_aa_stat} on page \pageref{eq_H_theta_aa_aa} of this paper.}  Mekler et al.\ use assumption \eqref{eq_LSTOmega_2Laa} to show that if $G$ is abelian, uncountable, and $\aleph_2$-free---i.e., if all if its subgroups of size $<\aleph_2$ are free---then \emph{Duplicator} must have a winning strategy in the Ehrenfeucht-Fra\"iss\'{e} game of length $\omega_1$ on the pair $(G,F)$ whenever $F$ is an uncountable free abelian group.  The other cases---e.g., where $F$ is countable, or when $G$ is not $\aleph_2$-free---are easily determined in ZFC alone.

We strengthen Theorem \ref{thm_MeklerEtAl} in two ways:  
\begin{enumerate}
 \item We weaken the assumption  \eqref{eq_LSTOmega_2Laa} to the Fuchino-Usuba~\cite{FuchinoUsuba} principle $\textbf{RP}_{\textbf{internal}}$ that, unlike \eqref{eq_LSTOmega_2Laa}, follows from Martin's Maximum (\cite{Cox_RP_IS}, \cite{Cox_Sakai_DRP}).  This improvement is due to a simple observation: that a certain sentence in Stationary Logic used in \cite{MR1191613} is equivalent, over $\aleph_1$-generated modules, to a different syntactic form  (see Corollary \ref{cor_Aleph_1_gen_CFilt} and subsequent discussion).

 \item\label{item_Decon}  Given a module $M$ and a collection $\mathcal{C}$ of countably-presented modules, we introduce the \textbf{$\boldsymbol{\mathcal{C}}$-filtration game of length $\boldsymbol{\omega_1}$ on $\boldsymbol{M}$} (Definition \ref{def_FiltGame}).  Player 2 attempts to prove that $M$ is ``$\mathcal{C}$-filtered" (see below), while Player 1 tries to prevent her from doing so; Player 2 wins if she lasts $\omega_1$ rounds.  These games have many interesting properties:
 \begin{enumerate}
  \item  $\text{RP}_{\text{internal}}$ implies their determinacy (if $\mathcal{C}$ has some nice quotient behavior, see Theorem \ref{thm_Cox_MM_Det}), and their determinacy---in the case where the ring is $\mathbb{Z}$ and $\mathcal{C} = \{ \mathbb{Z} \}$---implies the conclusion of Mekler et al.'s Theorem \ref{thm_MeklerEtAl} (see Corollary \ref{cor_DetImpliesDet}).
  \item Winning strategies of Players 1 and 2 are related, respectively, to Abstract Elementary Classes and ``potential" membership in the class of $\mathcal{C}$-filtered modules (see below).
\end{enumerate}

\end{enumerate}

``$M$ is $\mathcal{C}$-filtered" means that there exists a $\subseteq$-increasing and continuous sequence $\langle M_\xi \ : \ \xi < \eta \rangle$ with union $M$, such that $M_0 = \{ 0 \}$, and each quotient of the form $M_{\xi+1}/M_\xi$, is isomorphic to a  member of $\mathcal{C}$ (but $M_\xi$ for $\xi > 0$ is \emph{not} required to be a member of $\mathcal{C}$).  This is a significant weakening of the assertion that $M$ is a direct sum of modules from $\mathcal{C}$, and has become ubiquitous in approximation theory since the proof of the Flat Cover Conjecture around the year 2000 (\cite{MR1832549}, \cite{MR1798574}, \cite{MR3010854}).  The key fact (\cite{MR2822215}) is that if there exists a cardinal $\kappa$ such that every member of $\mathcal{C}$ is $<\kappa$-presented,\footnote{In which case the class of $\mathcal{C}$-filtered modules is called \textbf{$\boldsymbol{\kappa}$-deconstructible}.} then the class of all $\mathcal{C}$-filtered modules is a ``precovering" class, which allows one to replicate many of the constructions from classical homological algebra ``relative" to the class of $\mathcal{C}$-filtered modules.

\begin{theorem}\label{thm_Cox_MM_Det}
Assume the stationary reflection principle $\text{RP}_{\text{internal}}$.  Assume $R$ is a ring of size at most $\aleph_1$.  Then for any ``quotient-hereditary"\footnote{See Definition \ref{def_QuotHered} for the meaning of ``$\mathcal{C}$ is quotient-hereditary".} collection $\mathcal{C}$ of countably-presented $R$-modules, the $\mathcal{C}$-Filtration Game of length $\omega_1$ on any $R$-module is determined.  

In particular,\footnote{See Lemma \ref{lem_HeredRingQuotHered} for why this is a special instance of the first part of the theorem.} if $R$ is a hereditary ring of size at most $\aleph_1$ and either $\mathcal{C}=\{ R \}$ or $\mathcal{C}=$ the collection of countably-generated projective modules,\footnote{For projective modules, countably-generated and countably-presented are equivalent (Section \ref{sec_Prelims}).} then $\mathcal{C}$-Filtration Games of length $\omega_1$ are always determined.  
\end{theorem}

If $\mathcal{C}$ is the collection of countably-generated projective modules, we often refer to the $\mathcal{C}$-filtration games as \textbf{Projective Filtration Games}, since (by Kaplansky~\cite{MR0100017}) the class of projective modules is exactly the class of $\mathcal{C}$-filtered modules.  Similarly, if $\mathcal{C}=\{ R \}$, we often refer to the $\mathcal{C}$-filtration games as \textbf{Free Filtration Games}.

%\noindent We also show in Section \ref{SubSec_DB_EF} that determinacy of Filtration Games of length $\omega_1$ (which are played on a single module) implies the determinacy of certain Ehrenfeucht-Fra\"iss\'{e} games of length $\omega_1$ (which are played on pairs of modules).  It follows that the conclusion of Theorem \ref{thm_Cox_MM_Det} implies the conclusion of Theorem \ref{thm_MeklerEtAl}, but, as noted above, the assumption we use ($\text{RP}_{\text{internal}}$) is weaker than \eqref{eq_LSTOmega_2Laa}. 

%Also, it is not entirely clear which version of L(aa) is being used in the proof of Theorem 10 and Proposition 5 in \cite{MR1191613}, but it appears to be something like the infinitary version $L_{\omega_1,\omega_1}(aa)$, presumably so that one can express ``$Z$ is free" in the language of groups when $Z$ is a countable abelian group.  The corresponding statement  \eqref{eq_LSTOmega_2Laa} for $L_{\omega_1,\omega_1}(aa)$ implies the Continuum Hypothesis.

Winning strategies for \emph{Duplicator} in Ehrenfeucht-Fra\"iss\'{e} games are closely related to the notion of \emph{potential isomorphisms} (i.e., isomorphisms introduced by forcing); this was first investigated in Nadel-Stavi~\cite{MR462942}.  This connection shows up in Filtration Games too; e.g., Player 2 has a winning strategy in the Projective Filtration Game of length $\omega_1$ on $M$ if and only if $M$ is \textbf{$\boldsymbol{\sigma}$-closed potentially projective}, meaning that $M$ is projective in some $\sigma$-closed forcing extension of the universe.  This is discussed in Section \ref{sec_FiltGameDef}.

The notion of $\sigma$-closed potential projectivity arises in another context too.  \v{S}aroch-Trlifaj~\cite{MR4186456} proved that if $\kappa$ is a strongly compact cardinal and $R$ is a ring of cardinality less than $\kappa$, then the direct limit of any $<\kappa$-directed system of projective $R$-modules is projective.  It is natural to ask what fragment of that result can possibly hold for, say, $\kappa = \aleph_n$ ($n \ge 1$).  We cannot hope to prove that projective modules are closed under $<\aleph_n$-directed limits for such $n$, even for $R=\mathbb{Z}$, because of well-known ZFC results about almost free groups (\cite{MR1249391}).  However, we \emph{can} consistently get a version of their theorem (for $n = 2$) if we replace ``projective" with ``$\sigma$-closed potentially projective":  
\begin{theorem}\label{thm_Cox_ClosedDL}
Assume $\text{RP}_{\text{internal}}$.  Suppose $R$ is a ring of size at most $\aleph_1$ and $\mathcal{C}$ is a collection of countably-presented $R$-modules that is quotient-hereditary.  Then the class of $\sigma$-closed potentially $\mathcal{C}$-filtered modules is closed under $<\aleph_2$-directed limits.

In particular:  if $R$ is a countable, hereditary ring, then the class of $\sigma$-closed potentially projective $R$-modules is closed under $<\aleph_2$-directed limits. 
\end{theorem}

Finally, we give some applications of Filtration Games to \emph{Abstract Elementary Classes (AEC)}, which were introduced by Shelah~\cite{MR1033034}.  An AEC is a (possibly class-sized) partial order satisfying some axioms that generalize certain properties possessed by the elementary submodel ordering in many logics.  Given a ring $R$, a regular cardinal $\mu$, and a collection $\mathcal{C}$ of $<\mu$-presented $R$-modules, let $\boldsymbol{\Gamma^{\textbf{Filt}\big( \mathcal{C} \big)}_{\mu, \textbf{P1}}}$ denote the class of $R$-modules $M$ such that Player 1 has a winning strategy in the $\mathcal{C}$-filtration game of length $\mu$ on $M$ (the main case of interest is $\mu = \omega_1$).  Define $M \prec_R N$ to mean there is an injective $R$-module homomorphism from $M$ to $N$.\footnote{For the purposes of the following it would not matter if we placed the additional constraints that the embedding is pure, or even elementary in the language of $R$-modules.}   First we prove a ZFC theorem:

\begin{theorem}\label{thm_AlwaysAEC}
If $\mathcal{C}$ is quotient-hereditary, then the partial order 
\[
\Big( \Gamma^{\text{Filt}(\mathcal{C})}_{\mu,\text{P1}}, \prec_R \Big)
\]
is (always) an AEC, with L\"owenheim-Skolem number at most $\text{max}\Big( |R|, 2^{<\mu}\Big)$.

In particular, if $R$ is a hereditary ring, then the class of $R$-modules such that Player 1 has a winning strategy in the Projective Filtration Game of length $\omega_1$ is an AEC. 
\end{theorem}

Consider the special case $\mu = \omega_1$.  If $\text{RP}_{\text{internal}}$ holds and $R$ and $\mathcal{C}$ satisfy the hypothesis of Theorem \ref{thm_Cox_MM_Det}, then the determinacy ensures that $\Gamma^{\text{Filt}(\mathcal{C})}_{\omega_1, \text{P1}}$ is the same as the class of modules such that Player 2 does \textbf{not} have a winning strategy in the $\mathcal{C}$-filtration game of length $\omega_1$.  But asserting that the latter class is an AEC turns out to be independent of ZFC:
\begin{theorem}\label{thm_AnswerBaldwin}
The class of abelian groups $G$ for which Player 2 does \textbf{not} have a winning strategy in the Free Filtration game of length $\omega_1$ on $G$, ordered by the subgroup (or pure subgroup, or elementary subgroup) relation, is an AEC with L\"owenheim-Skolem number $\aleph_1$ in some models of set theory, but is not even an AEC in others.
\end{theorem}

Section \ref{sec_Prelims} covers preliminaries.  Section \ref{sec_Hill_and_related} provides the connection between stationary sets and potentially $\mathcal{C}$-filtered modules, including a useful consequence of \emph{Hill's Lemma}.  Section \ref{sec_FiltGames} introduces Filtration Games, the Dual Basis Game, and analyzes the relationship between these games and Ehrenfeucht-Fra\"iss\'{e} games.  Section \ref{sec_FiltGames} also includes the proof of Theorem \ref{thm_AlwaysAEC}.   Section \ref{sec_MainThms} proves an ``$\aleph_2$-compactness" result (under $\text{RP}_{\text{internal}}$) that is, in turn, used to prove Theorems \ref{thm_Cox_MM_Det},  \ref{thm_Cox_ClosedDL}, and \ref{thm_AnswerBaldwin}.  Section \ref{sec_Conclusion} lists some open problems.

\section{Preliminaries}\label{sec_Prelims}

Unless otherwise indicated, our terminology agrees with Jech~\cite{MR1940513} (for the set-theoretic background) and Eklof-Mekler~\cite{MR1914985} and G\"obel-Trlifaj~\cite{MR2985554} (for the algebra background).  By ``$R$-module" we will officially mean a left $R$-module.  If $X$ is a subset of an $R$-module $M$, $\langle X \rangle^M_R$ denotes the $R$-submodule of $M$ generated by $X$; the $M$ and $R$ decorations will be omitted when it is clear from the context.  A module $M$ is \textbf{$\boldsymbol{\mu}$-generated} if $M = \langle X \rangle$ for some $X \subseteq M$ with $|X|=\mu$; equivalently, there exists a short exact sequence $0 \to K \to F \to M \to 0$ where $F$ is free of rank $\mu$.  $M$ is \textbf{$\boldsymbol{\mu}$-presented} if there exists a short exact sequence $0 \to K \to F \to M \to 0$ where $F$ is free of rank $\mu$, and $K$ is $\mu$-generated.  A module $P$ is \textbf{projective} if it is a direct summand of some free module.  Equivalently, $P$ is projective if every short exact sequence of the form 
\[
0 \to K \to M \to P \to 0
\]
 splits, in which case $M \simeq K \oplus P$.  In particular, if $P$ is projective and $\mu$-generated as witnessed by a short exact sequence
 \[
0 \to K \to F \to P \to 0 
 \]
where $F$ is free of rank $\mu$, then $F \simeq K \oplus P$, and hence $K \simeq F/P$.  So $K$ is $\mu$-generated, yielding:
\begin{fact}\label{fact_ProjCGCP}
If $P$ is projective and $\mu$-generated, then $P$ is $\mu$-presented.
\end{fact}

Occasionally we will refer to the notion of a pure embedding:  an $R$-submodule $M$ of $N$ is called a \textbf{pure submodule} if for every (finite) matrix $A$ with entries from $R$, and every vector $\vec{b}$ of elements from $M$, if $A\vec{x}=\vec{b}$ has a solution in $N$, then it has a solution in $M$.  An embedding $\rho: M \to N$ is called a pure embedding if its image is a pure submodule of $N$.

\subsection{Dual Basis characterization of projectivity}
Recall that projective modules are, by definition, direct summands of free modules.  Several classical characterizations of projectivity appear in Lam~\cite{MR1653294}.  For example, a module is projective if and only if it has a \textbf{dual basis}.  A dual basis for an $R$-module $M$ is a pair
\[
\mathcal{D} = \Big( B, \big( f_b \big)_{b \in B} \Big)
\]
such that $B \subseteq M$, each $f_b: M \to R$ is $R$-linear, and for all $x \in M$, 
\[
\text{sprt}_{\mathcal{D}}(x):= \{ b \in B \ : \ f_b(x) \ne 0 \}
\]
 is finite and 
 \[
 x = \sum_{b \in \text{sprt}(x)} f_b(x) b.
 \]
 
\noindent Existence of a dual basis for $M$ is equivalent to projectivity of $M$ (see e.g.\ Lam~\cite{MR1653294}).  The dual basis characterization of projectivity yields the following nice consequence, that projectivity is absolute for countably-generated modules:

\begin{lemma}\label{lem_Proj_Abs}
Let $\mathcal{H} \subset \mathcal{H}'$ be transitive $\text{ZFC}^-$ models.  Then for any ring $R \in \mathcal{H}$ and any $R$-module $M \in \mathcal{H}$ such that $\mathcal{H} \models$ ``$M$ is countably generated",
\[
\mathcal{H} \models \ M \text{ is projective} \iff \ \mathcal{H}' \models M \text{ is projective.}
\]
\end{lemma}
\begin{proof}
The forward direction holds regardless of the size of $M$ in $\mathcal{H}$, because projectivity of $M$ is witnessed by a dual basis for $M$, which is easily upward absolute from $\mathcal{H}$ to $\mathcal{H}'$ (alternatively, $\mathcal{H}$ sees that $M$ is a direct summand of a free module, and this is clearly upward absolute to $\mathcal{H}'$).

For the $\Leftarrow$ direction, in $\mathcal{H}$ fix a countable set $Z$ generating $M$, and in $\mathcal{H}$ let $T$ be the tree of finite attempts to build a dual basis for $M = \langle Z \rangle^M_R$.  In other words, in $\mathcal{H}$ fix an enumeration $\{ z_n \ : \ n < \omega \}$ of $Z$, and let $T$ be the tree whose nodes are triples
\[
t = \Big( n^t, B^t, (f^t_b)_{b \in B^t} \Big)
\]
such that
\begin{enumerate}
 \item $n^t \in \omega$;
 \item $B^t$ is a finite subset of $M$;
% \item $X^t$ is a finite superset of $\{ z_n \ : \ n < n_t \}$;
 \item Each $f^t_b$ is a function from 
 \[
  \{ z_n \ : \ n < n^t \}  \to R
  \]
  that lifts to an $R$-module homomorphism from $\langle \{ z_n \ : \ n < n^t \} \rangle \to R$;
 \item For each $n < n_t$,
 \[
 z_n = \sum_{b \in B^t} f_b(z_n) b.
 \]

\end{enumerate}
with tree ordering as follows:  $t > s$ if $n_t > n_s$, $B^t \supseteq B^s$, the $f^t_b$'s extend the $f^s_b$'s for all $b \in B^s$; and, importantly, for all $n < n^s$, 
\[
\{ b \in B^t \ : \ f_b^t(z_n) \ne 0   \} \subseteq B^s.
\]
The last requirement ensures that branches through the tree won't make the support of any member of $Z$ infinite.

Now since $M$ is projective in $\mathcal{H}'$, there is a dual basis for $M$ in $\mathcal{H}'$, and this dual basis can be used to recursively build an infinite branch through $T$ in $\mathcal{H}'$.  Since $T \in \mathcal{H}$, absoluteness of wellfoundedness between transitive $\text{ZFC}^-$ models ensures that there is an infinite branch through $T$ in $\mathcal{H}$ too.  And any infinite branch through $T$ yields a dual basis for $M$.\footnote{It is clear that a cofinal branch through $T$ yields a set $B \subseteq M$ and homomorphisms $f_b: \langle Z \rangle = M \to R$ such that for every $\boldsymbol{z \in Z}$, $\text{sprt}(z):= \{ b \in B \ : \ f_b(z) \ne 0 \}$ is finite, and $z = \sum_{\text{sprt}(z)} f_b(z) b$.  It is then routine to check that the latter also holds for all $x \in \langle Z \rangle$; i.e., if $x \in \langle z_1, \dots, z_k \rangle$, then $\text{sprt}(x) \subseteq \bigcup_{i \le k} \text{sprt}(z_i)$, and $x = \sum_{b \in \bigcup_{i \le k} \text{sprt}(z_i)} f_b(x)b$. }
\end{proof}

An easier argument, just using  closure of $\mathcal{H}$ in $\mathcal{H}'$ to absorb the dual basis, yields the following.  Note that for $\mu = \omega_1$ the following lemma is superfluous, since Lemma \ref{lem_AbsCfilt} got the same conclusion \textbf{without} assuming $<\omega_1$-closure of $\mathcal{H}$ in $\mathcal{H}'$.
\begin{lemma}\label{lem_ProjAbsMuClosed}
Let $\mathcal{H} \subset \mathcal{H}'$ be transitive $\text{ZFC}^-$ models, and assume $\mu$ is an infinite regular cardinal from the point of view of $\mathcal{H}$ such that $\mathcal{H}$ is closed under $<\mu$-length sequences from $\mathcal{H}'$.  Then for any ring $R \in \mathcal{H}$ and any $R$-module $M \in \mathcal{H}$ such that $\mathcal{H} \models$ ``$M$ is $<\mu$-generated",
\[
\mathcal{H} \models \ M \text{ is projective} \iff \ \mathcal{H}' \models M \text{ is projective.}
\]

\end{lemma}

\subsection{Filtrations}\label{subsec_Filtrations}

If $X$ is a set, a \textbf{filtration of $\boldsymbol{X}$} is a sequence $\langle X_i \ : \ i < \eta \rangle$, for some ordinal $\eta$, such that
\begin{itemize}
 \item $X_i \subseteq X_j$ whenever $i < j$ (not necessarily proper containment);
 \item $X_i = \bigcup_{k < i} X_k$ whenever $i$ is a limit ordinal; and
 \item $X = \bigcup_{i < \eta} X_i$.
\end{itemize}

\noindent \textbf{Note:} we do \emph{not} require that $\eta$ is a cardinal, nor that each $X_i$ is of smaller cardinality than $X$.  Our definition above adheres to the common notion of filtration from the module theory literature (e.g., G\"obel-Trlifaj~\cite{MR2985554}), though some set theory literature (including previous papers of the author) have included those additional assumptions in the definition of \emph{filtration}.  We will often want filtrations to have these additional properties, however, as discussed in Section \ref{sec_Hill_and_related}.

Let $\mathcal{C}$ be a class of $R$-modules.  An $R$-module $M$ is said to be $\boldsymbol{\mathcal{C}}$\textbf{-filtered} if there exists a filtration $\vec{M} = \langle M_i \ : \ i < \eta \rangle$ of $M$ such that $M_0 = \{ 0 \}$ and for all $i < \eta$, if $i+1 < \eta$ then
\[
\frac{ M_{i+1} }{M_i } \text{ is isomorphic to an element of } \mathcal{C}.
\]

\noindent We say that \textbf{$\boldsymbol{\mathcal{C}}$ is closed under transfinite extensions} if every $\mathcal{C}$-filtered module is an element of $\mathcal{C}$.  Common examples of classes closed under transfinite extensions are the class of projective modules, the class of free modules, and the roots of $\text{Ext}(-,N)$ for any fixed module $N$ (``Eklof's Lemma"; see \cite{MR453520}).

\begin{lemma}\label{lem_ConcatenateFilt}
If $\langle M_i \ : \ i < \eta \rangle$ is a $\mathcal{C}$-filtration of $M$, then for all $i < j < \eta$, $M_j/M_i$ is $\mathcal{C}$-filtered.
\end{lemma}
\begin{proof}
Fix $i < j$ where $j < \eta$.  Then
\[
\Big\langle M_k/M_i \ : \ k \in [i,j) \Big\rangle
\]
is a filtration of $M_j/M_i$ if $j$ is a limit ordinal, and \[
\Big\langle M_k/M_i \ : \ k \in [i,j] \Big\rangle
\]
is a filtration of $M_j/M_i$ if $j$ is a successor ordinal.  And each has consecutive factors that are (isomorphic to an element) in $\mathcal{C}$, since $\frac{M_{k+1}/M_i}{M_k/M_i} \simeq \frac{M_{k+1}}{M_k}$.
\end{proof}

\subsection{Stationary sets}

If $X$ is a set and $\mu$ is a regular uncountable cardinal, $\wp_\mu(X)$ denotes the set
\[
\{ Z \subset X \ : \ |Z| < \mu  \}.
\]
We sometimes write $[X]^\omega$ instead of $\wp_{\omega_1}(X)$.  We also use $[X]^{<\omega}$ to denote all finite subsets of $X$.  A set $C$ is called  \textbf{closed and unbounded (club) in $\boldsymbol{\wp_\mu(X)}$} if and only if $C \subseteq \wp_\mu(X)$ and:
\begin{itemize}
 \item $C$ is $\subseteq$-cofinal in $\wp_\mu(X)$; and
 \item If $\vec{Z}$ is a $\subseteq$-increasing chain of members of $C$ of length strictly less than $\mu$, then the union of the chain is in $C$. 
\end{itemize}
A set $S$ is called \textbf{stationary in $\boldsymbol{\wp_\mu(X)}$} if it meets every closed unbounded subset of $\wp_\mu(X)$.  We need a lemma due to Kueker (see \cite{MR1940513}):
\begin{lemma}\label{lem_Kueker}
Assume $X$ is a set and $X \supseteq \mu$.  Then $D$ contains a club in $\wp_\mu(X)$ if and only if there exists an $F:[X]^{<\omega} \to X$ such that
\[
C_{F,\mu,X}:= \{ z \in \wp_\mu(X) \ : \  z \cap \mu \text{ is an ordinal and } z \text{ is closed under } F \}
\]
is contained in $D$.  For $\mu = \omega_1$, the ``$z \cap \mu$ is an ordinal" is unnecessary in the definition of $C_{F,\mu,X}$.\footnote{For $\mu \ge \omega_2$ this technical requirement is essential, due to the possibility of \emph{Chang's Conjecture} holding (see \cite{MattHandbook}).}  
\end{lemma}

For a cardinal $\theta$, $H_\theta$ denotes the collection of sets of hereditary cardinality $<\theta$.  If $\theta$ is regular and uncountable, then $H_\theta$ is transitive and $(H_\theta,\in)$ satisfies all axioms of ZFC except possibly the powerset axiom.

\begin{corollary}\label{cor_ClubRef}
Suppose $\mu$ is a regular uncountable cardinal, $X$ is a set of cardinality at least $\mu$, $C$ is a closed unbounded subset of $\wp_\mu(X)$, $\theta$ is a regular cardinal with $(X,C,\mu) \in H_\theta$, and $W$ is an elementary substructure of $(H_\theta, \in)$ such that $(X,C,\mu) \in W$ and $\mu \subseteq W$.  Then $C \cap \wp_\mu(W \cap X)$ contains a club in $\wp_\mu(W \cap X)$.
\end{corollary}
\begin{proof}
We can without loss of generality assume that $X \supseteq \mu$, since any elementary substructure of $(H_\theta,\in)$ that has $X$ as an element will see a bijection between $X$ and an ordinal.  Then by Lemma \ref{lem_Kueker}, there is an $F:[X]^{< \omega} \to X$ such that $C_{F,\mu, X} \subseteq  C$. By elementarity of $W$ in $H_\theta$ and the fact that $C \in W$, we can assume $F \in W$.  Then by elementarity of $W$, 
\begin{equation}
\label{eq_WcapXclosedF}
W \cap X \text{ is closed under } F.
\end{equation}
In other words, $F|W:=F \restriction [W \cap X]^{<\omega}$ maps into $W \cap X$, and hence the set $C_{F|W,\mu,W \cap X}$ contains a closed unbounded subset of $\wp_\mu(W \cap X)$.  And, clearly, this set is contained in $C_{F,\mu,X}$ (and therefore in $C$).
\end{proof}

We also need another standard corollary of Kueker's characterization of club sets (see \cite{MattHandbook}):
\begin{lemma}\label{lem_CharClubElemSub}
Suppose $\mu$ is a regular uncountable cardinal, $X$ is any set, and $D$ contains a closed unbounded subset of $\wp_\mu(X)$.  Then for all regular $\theta$ such that $X,D \in H_\theta$, and all $W \prec (H_\theta,\in,X,D)$:   if $|W|<\mu$ and $W \cap \mu$ is an ordinal, then $W \cap X \in D$.
\end{lemma}

\section{Stationary sets and potentially filtered modules}\label{sec_Hill_and_related}

In this section we investigate the notion of being $\mathcal{C}$-filtered in certain forcing extensions.  To motivate the main technical lemmas of this section, suppose $\mathcal{C}$ is a collection of countably presented modules, and suppose $M$ is a module that is $\aleph_1$-generated and $\mathcal{C}$-filtered.  Recall that ``$M$ is $\mathcal{C}$-filtered" means that there exists a filtration
\[
\vec{M}=\langle M_i \ : \ i < \eta \rangle
\]
of $M$ such that $M_0=\{0\}$ and $M_{i+1}/M_i$ is (isomorphic to a module) in $\mathcal{C}$ for all $i$.  But the definition does not require that $\eta$ is a cardinal, much less that $\eta = \omega_1$ or that the $M_i$'s themselves are countably generated.  On the other hand, since $M$ has a generating set of size $\aleph_1$, there exists a $\subseteq$-increasing and continuous sequence $\vec{Z}=\langle Z_i \ : \ i < \omega_1 \rangle$ of countable subsets of $M$, whose union generates $M$.  For various reasons, it would be nice to know that we can replace $\vec{M}$ with another $\mathcal{C}$-filtration that looks more like $\vec{Z}$; in particular, which has length exactly $\omega_1$ and whose entries are countably generated.  The following lemma, which follows from the highly versatile \emph{Hill Lemma}, will allow us to do exactly that (see also Corollary \ref{cor_Aleph_1_gen_CFilt}):

\begin{lemma}\label{lem_HillClub}
Suppose $\mu$ is a regular uncountable cardinal and $\mathcal{C}$ is a class of $<\mu$-presented modules.  Suppose $M$ is $\mathcal{C}$-filtered.  Then there is a closed unbounded subset $D$ of $\wp_{\mu}(M)$ such that whenever $Z \subseteq Z'$ are both in $D$, 
\[
\langle Z \rangle \text{ and } \frac{\langle Z' \rangle}{\langle Z \rangle} \text{ are } \mathcal{C} \text{-filtered (via filtrations of length $<\mu$).}
\]

\end{lemma}
\begin{proof}
Let $\vec{M} = \langle M_i \ : \ i < \eta \rangle$ witness that $M$ is $\mathcal{C}$-filtered.  By the Hill Lemma (Theorem 7.10 of G\"obel-Trlifaj~\cite{MR2985554}), there is a family $\mathcal{F}$ of submodules of $M$ such that  
\begin{enumerate*}[label=(\roman*)]
\item\label{item_ContainsFilt} $\mathcal{F}$ contains $\{ M_i \ : \ i < \eta \}$;
\item\label{item_ClosedUnderSums} $\mathcal{F}$ is closed under arbitrary sums and intersections; 
\item\label{item_QuotientsCfiltered} whenever $N \subset N'$ are both in $\mathcal{F}$ then $N'/N$ is $\mathcal{C}$-filtered; and
\item\label{item_Cof}  whenever $N \in \mathcal{F}$ and $X$ is a $<\mu$-sized subset of $M$, there is an $N' \in \mathcal{F}$ containing $\langle N \cup X \rangle$ such that $N'/N$ is $<\mu$-presented.
\end{enumerate*}

Set 
\[
D:= \{ Z \subseteq M \ : \ |Z| < \mu \text{ and } \langle Z \rangle \in \mathcal{F} \},
\]
i.e., $D$ is the collection of all subsets of $M$ of size $<\mu$ that generate an element of $\mathcal{F}$.  We claim that $D$ is as desired.  Note first that $D$ is nonempty, because $M_1 = M_1/\{ 0 \}$ is (isomorphic to) an element of $\mathcal{C}$ and hence $<\mu$-presented (and an element of $\mathcal{F}$, by \ref{item_ContainsFilt}).  To see closure, suppose $\langle Z_k \ : \ k < \zeta \rangle$ is a $\subseteq$-increasing sequence from $D$, where $\zeta < \mu$.  Part \ref{item_ClosedUnderSums} ensures that $\bigcup_{k < \zeta} \langle Z_k \rangle = \left\langle \bigcup_{k < \zeta} Z_k \right\rangle$ is an element of $\mathcal{F}$; and since $\mu$ is regular and $\zeta < \mu$, $\bigcup_{k < \zeta} Z_k$ has size $<\mu$.  Hence, $\bigcup_{k < \zeta} Z_k \in D$.  To see that $D$ is $\subseteq$-cofinal in $\wp_\mu(M)$, fix a $<\mu$-sized subset $X$ of $M$.  By \ref{item_Cof}, (taking $N = M_0 = \{ 0 \}$) there is some $N' \in \mathcal{F}$ such that $X \subseteq N'$ and $N'/ \{ 0 \} \simeq N'$ is $<\mu$- presented; say $Z$ is a $<\mu$-sized set of generators for $N'$.  Then $N' = \langle Z \rangle = \langle X \cup Z \rangle$, and hence $X \cup Z \in D$.  Hence, $D$ is closed and unbounded in $\wp_\mu(M)$.

Now suppose $Z \subset Z'$ are both in $D$; then $\langle Z \rangle \subset \langle Z' \rangle$ are both in $\mathcal{F}$, and so by \ref{item_QuotientsCfiltered}, $\langle Z' \rangle / \langle Z \rangle$ is $\mathcal{C}$-filtered; say by 
\[
\big\langle U_\xi / \langle Z \rangle \ :  \ \xi < \zeta \big\rangle.
\]
We can without loss of generality assume there are no superfluous indices; i.e.\, that $U_{\xi}$ is a proper subset of $U_{\xi+1}$ for all $\xi < \zeta$.  Since the union of the $U_\xi$'s equals $\langle Z' \rangle$, $|Z'|< \mu$, and $\mu$ is regular, it follows that $\zeta < \mu$.  So it is a $\mathcal{C}$-filtration of $\langle Z' \rangle / \langle Z \rangle$ of length $<\mu$.

Now \ref{item_QuotientsCfiltered}, together with the fact that $M_0 = \{ 0 \} \in \mathcal{F}$, ensure that every element of $\mathcal{F}$ is $\mathcal{C}$-filtered; so, in particular, $\langle Z \rangle$ is $\mathcal{C}$-filtered whenever $Z \in D$.

\end{proof}

%NOT SURE IF THIS IS TRUE AND I THINK WE DON'T NEED IT  We remark that Lemma \ref{lem_HillClub} still holds if one only requires $\mathcal{C}$ to consist of $<\mu$-generated modules (rather than $<\mu$-presented modules).  This is because there is a version of the Hill Lemma (cited in the proof) for the $<\mu$-generated setting, too.  However, since we could find no citation for that version of the Hill Lemma, and since the distinction is not important for our applications, we require $\mathcal{C}$ to consist of $<\mu$-presentable modules.

\begin{corollary}\label{cor_Aleph_1_gen_CFilt}
Suppose $\mu$ is regular and uncountable, $\mathcal{C}$ is a class of $<\mu$-presented modules, and $M$ is a $\mu$-generated module.  The following are equivalent:
\begin{enumerate}
 \item\label{item_C_filtered} $M$ is $\mathcal{C}$-filtered.
 \item\label{item_C_filtered_omega1} $M$ is $\mathcal{C}$-filtered by a filtration of (ordinal) length at most $\mu$, all of whose entries are $<\mu$-presented.
 \item\label{item_aa_aa_M} There is a closed unbounded $D \subseteq \wp_\mu(M)$ such that whenever $Z \subseteq Z'$ are both in $D$, $\langle Z' \rangle / \langle Z \rangle$ is $\mathcal{C}$-filtered, via a filtration of length strictly less than $\mu$.
 \item\label{item_aa_stat_M}  there is a closed unbounded set $D \subseteq \wp_\mu(M)$ such that for every $Z \in D$, 
\begin{equation*}
\begin{split}
S_Z:= \{ Z' \in D \ : \ & Z \subset Z' \text{ and } \langle Z' \rangle / \langle Z \rangle \text{ is } \mathcal{C} \text{-filtered} \\ 
& \text{via a $<\mu$-length filtration} \}
\end{split}
\end{equation*}
is stationary in $\wp_\mu(M)$.

\end{enumerate}
\end{corollary}
\begin{proof}
Let $X$ be an $\mu$-sized generating set for $M$, and fix an enumeration $\{x_i \ : \ i < \mu \}$ of $X$.  

For the \eqref{item_C_filtered} implies \eqref{item_C_filtered_omega1} direction, suppose $M$ is $\mathcal{C}$-filtered.  By Lemma \ref{lem_HillClub} there is a closed unbounded $D \subseteq \wp_{\mu}(M)$ such that whenever $Z \subset Z'$ are both in $D$, $\langle Z' \rangle / \langle Z \rangle$ is $\mathcal{C}$-filtered.  Using that $D$ is closed unbounded in $\wp_{\mu}(M)$, recursively build a $\subseteq$-increasing and continuous chain $\langle Z_i \ : \ i < \mu \rangle$ of $<\mu$-sized subsets of $M$ such that for all $i < \mu$:
\begin{itemize}
 \item $x_i \in Z_{i+1}$;
 \item $Z_i \in D$
\end{itemize}
Now $\Big\langle \langle Z_i \rangle \ : \ i < \mu \Big\rangle$ is a filtration of $M$, but probably there are many $i$'s such that $\langle Z_{i+1} \rangle/ \langle Z_i \rangle$ is not in $\mathcal{C}$, but merely $\mathcal{C}$-filtered.\footnote{Of course if $\mathcal{C}$ is closed under transfinite extensions, or even just countable transfinite extensions, it does not matter.}  However, since $Z_i \subset Z_{i+1}$ are both in $D$, there is a $\mathcal{C}$-filtration 
\begin{equation}\label{eq_CfiltQuot}
\left\langle U^i_\xi / \langle Z_i \rangle \ : \ \xi < \zeta_i \right\rangle
\end{equation}
of the quotient $\langle Z_{i+1} \rangle/ \langle Z_i \rangle$ such that $\zeta_i < \mu$.  Then concatenating all sequences of the form
\[
\langle U^i_\xi \ : \ \xi < \zeta_i \rangle
\]
across all $i < \mu$ yields a $\mathcal{C}$-filtration of $M$ of length at most $\mu$; it is a $\mathcal{C}$ filtration because for each $i$ and $\xi$ such that $\xi+1 < \zeta_i$,
\[
\frac{U^i_{\xi+1}}{U^i_\xi} \simeq \frac{U^i_{\xi+1} / \langle Z_i \rangle}{U^i_{\xi} / \langle Z_i \rangle}
\]
which is in $\mathcal{C}$ because \eqref{eq_CfiltQuot} is a $\mathcal{C}$-filtration.

The \eqref{item_C_filtered_omega1} implies \eqref{item_aa_aa_M} direction follows from Lemma \ref{lem_HillClub}.  The \eqref{item_aa_aa_M} implies \eqref{item_aa_stat_M} is trivial.  Finally, assume \eqref{item_aa_stat_M}; we want to build a $\mathcal{C}$-filtration for $M$.  Fix a club $D \subseteq \wp_\mu(M)$ and a stationary $S_Z \subseteq \wp_\mu(M)$ for each $Z \in D$, as in the assumptions of \eqref{item_aa_stat_M}.  Recursively construct a $\subseteq$-increasing sequence $\langle Z_i \ : \ i < \mu \rangle$ such that:
\begin{itemize}
 \item $x_i \in Z_{i+1} \in S_{Z_i} \cap D$ for all $i < \mu$, provided that $Z_i \in D$ (otherwise halt the construction; note that if $S_{Z_i}$ is defined, then $S_{Z_i} \cap D$ is stationary in $\wp_\mu(M)$);
 \item $Z_i = \bigcup_{k < i}Z_k$ for all limit ordinals $i$.
\end{itemize}
\noindent Note that if $i$ is a limit ordinal below $\mu$ and $Z_k \in D$ for all $k < i$, then $Z_i \in D$ by closure of $D$ in $\wp_\mu(M)$.  So the construction never breaks down at any ordinal before $\mu$.  We have constructed a filtration $\vec{Z} = \Big\langle \langle Z_i \rangle \ : \ i < \mu \Big\rangle$ of $M$, such that for all $i < \mu$, $\langle Z_{i+1} \rangle / \langle Z_i \rangle$ is $\mathcal{C}$-filtered.  And this filtration of $\langle Z' \rangle / \langle Z \rangle$ is of length less than $\mu$, assuming no redundant entries.
\end{proof}

We remark that parts \eqref{item_aa_aa_M} and \eqref{item_aa_stat_M} in the statement of Corollary \ref{cor_Aleph_1_gen_CFilt} can be rephrased in terms of \emph{Stationary Logic} (see \cite{MR486629}), respectively, as
\begin{equation}\label{eq_H_theta_aa_aa}
\begin{split}
(H_\theta,\in, R,M, \mathcal{C} \cap H_\theta) \models &  \text{aa} Z \ \text{aa} Z' \ \langle Z \cap M \rangle \text{ and } \\
& \langle Z' \cap M  \rangle / \langle Z \cap M \rangle \text{ are } \mathcal{C} \text{-filtered} 
\end{split}
\end{equation}
and
\begin{equation}\label{eq_H_theta_aa_stat}
\begin{split}
(H_\theta,\in, R,M, \mathcal{C} \cap H_\theta) \models &  \text{aa} Z \ \textbf{stat} Z' \ \langle Z \cap M \rangle \text{ and } \\
& \langle Z' \cap M  \rangle / \langle Z \cap M \rangle \text{ are } \mathcal{C} \text{-filtered} 
\end{split}
\end{equation}
for any $\theta$ such that $(R,M) \in H_\theta$; here the interpretation of \emph{aa} and \emph{stat} depend on $\mu$, though the most interesting case for us is $\mu = \omega_1$.\footnote{Statements \eqref{eq_H_theta_aa_aa} and \eqref{eq_H_theta_aa_stat} involve some slight abbreviations; e.g. \eqref{eq_H_theta_aa_aa} should really say the following, where lowercase variables are first order:
\begin{equation*}
\begin{split}
(H_\theta,\in, R,M, \mathcal{C} \cap H_\theta) \models &   \text{aa} Z \ \text{aa} Z' \ \exists p  \ \exists p'  \ p = \langle Z \cap M \rangle, \ p' = \langle Z \cap M \rangle, \text{ and } p'/p \in \mathcal{C}.
\end{split}
\end{equation*}
}  The ability to replace the second \emph{aa} quantifier with a \emph{stat} quantifier is a key observation, since it allows us to prove Theorem \ref{thm_Cox_MM_Det} from a ``non-diagonal" version of stationary set reflection that, unlike assumption \eqref{eq_LSTOmega_2Laa} of Theorem \ref{thm_MeklerEtAl}, follows from Martin's Maximum.

Note that if $\mathcal{C}$ happens to be closed under transfinite extensions of length $<\mu$, then in the equivalences above, one can also require that $\langle Z \rangle \in \mathcal{C}$ for club-many $Z \in \wp_\mu(M)$.

\begin{definition}
Let $\mu$ be a regular (not necessarily uncountable) cardinal.  Given a class $\mathcal{C}$ modules, a module $M$ will be called \textbf{$\boldsymbol{<\mu}$-closed potentially $\boldsymbol{\mathcal{C}}$-filtered} if $M$ is $\mathcal{C}$-filtered in some $<\mu$-closed forcing extension of the universe.  We will say that $M$ is \textbf{$\boldsymbol{<\mu}$-closed potentially projective} if $M$ is projective in some $<\mu$-closed forcing extension of the universe.  For the case $\mu = \omega$, we will sometimes just say ``potentially $\mathcal{C}$-filtered" instead of ``$<\omega$ potentially $\mathcal{C}$-filtered", since every forcing poset is, trivially, $<\omega$ closed.
\end{definition}

\begin{remark}
The parameter $\mathcal{C}$ in the definition above refers to the actual class in the ground model, \textbf{not} to any interpretation in the forcing extension.  That is, ``$M$ is $<\mu$-closed potentially $\mathcal{C}$-filtered" officially means that there is a $<\mu$-closed poset $\mathbb{P}$ forcing ``$\check{M}$ is $\check{\mathcal{C}}$-filtered" (where $\mathcal{C}$ is a class in the ground model).  However, in some cases the interpretations will coincide, e.g.\ when $\mathcal{C}$ is the collection of countably-generated or countably presented projective modules; see Corollary \ref{cor_SameClasses} below.
\end{remark}

\begin{lemma}\label{lem_SigmaClosedLengthOmega1}
Let $\mu$ be regular uncountable and $\mathcal{C}$ be a collection of $<\mu$-presented modules.  If $M$ is $<\mu$-closed potentially $\mathcal{C}$-filtered, then there is a $<\mu$-closed forcing extension of $V$ in which $|M|\le \mu$ and there is a $\mathcal{C}$-filtration of $M$ of length at most $\mu$.
\end{lemma}
\begin{proof}
Let $W$ be a $<\mu$-closed extension of the ground model $V$ where $M$ is $\mathcal{C}$-filtered.  Let $W'$ be a further $<\mu$-closed forcing extension of $W$ where $|M|\le \mu$.  Clearly the $\mathcal{C}$-filtration of $M$ is upward absolute to $W'$, and every member of $\mathcal{C}$ is $<\mu$-presented in $W'$. Note that $W'$ is a $<\mu$-closed forcing extension of $V$.  And by Corollary \ref{cor_Aleph_1_gen_CFilt}, $W'$ has a $\mathcal{C}$-filtration of $M$ of length at most $\mu$.  
\end{proof}

%---we can sometimes make use of absoluteness arguments to arrange that the formula defining $\mathcal{C}$ is absolute between the ground model and the $\sigma$-closed forcing extension.  See, for example, Corollary \ref{}, which shows that the ground model is correct about projectivity for all countably presented modules in $\sigma$-closed forcing extensions.

\begin{lemma}\label{lem_CtblePresSigmaClosed}
Suppose $V \subset W$ are transitive ZFC models, $\mu$ is a cardinal in $V$, $V$ is $\mu$-closed in $W$ (i.e., $W \cap {}^\mu V \subset V$), and $R$ is a ring in $V$.  Then:
\begin{enumerate}
 \item\label{item_MuGenSubmod} For every fixed $R$-module $M$ in $V$, every $\mu$-generated submodule of $M$ in $W$ is also in $V$.
 \item\label{item_MuGenRelations} Every $R$-module in $W$ whose relations are $\mu$-generated is (isomorphic to) an element of $V$.
 \item\label{item_MuPresented} Every $\mu$-presented $R$-module in $W$ is (isomorphic to) an $R$-module in $V$ that is $\mu$-presented in $V$.
\end{enumerate}

\end{lemma}
\begin{proof}
For part \eqref{item_MuGenSubmod}, if $Z \in W$ is a $\mu$-sized subset of $M$, then $Z \in V$ by $\mu$-closure of $V$ in $W$, and hence $\langle Z \rangle^M_R$ is an element of $V$, since $Z$, $M$, and $R$ are in $V$.

For part \eqref{item_MuGenRelations}:  in $W$, suppose $N \simeq F/K$ where $F$ is free and $K$ is a $\mu$-generated submodule of $F$; say $K = \langle Z \rangle^F_R$ where $Z$ is (in $W$) a $\mu$-sized subset of $F$.  Note that $F \in V$, since the ring $R$ is in $V$ and $F$ is a direct sum of copies of $R$; then $Z \in V$ too, by the $\mu$-closure of $V$ in $W$.  Hence $K = \langle Z \rangle^F_R \in V$, and therefore $F/K \in V$.

For part \eqref{item_MuPresented}:  if the $F$ from the previous paragraph was also $\mu$-generated in $W$, then by the $\mu$-closure of $V$ in $W$, $F$ is also $\mu$-generated in $V$.  Hence the quotient $F/K$ is a $\mu$-presentation from the point of view of $V$ too.
\end{proof}

Our focus will often be on the case $\mu = \omega_1$, in which case we will say ``$\sigma$-closed potentially $\mathcal{C}$-filtered" instead of ``$<\omega_1$-closed potentially $\mathcal{C}$-filtered" (and similarly $\sigma$-closed instead of $<\omega_1$ closed forcing).

\begin{corollary}\label{cor_SameClasses}
For any ring $R$, and any $R$-module $M$, the following are equivalent:
\begin{enumerate}
 \item\label{item_PP} $M$ is $\sigma$-closed potentially projective;
 
% \item\label{item_Cp_filt} The class of $\sigma$-closed potentially $\mathcal{C}_p$-filtered modules, where $\mathcal{C}_p$ is the collection of all countably presented, projective modules.  

 \item\label{item_Cg_filt} $M$ is $\sigma$-closed potentially $\mathcal{C}$-filtered, where $\mathcal{C}$ is the collection of all countably generated, projective modules.\footnote{Which by Fact \ref{fact_ProjCGCP} is the same as the collection of all countably presented, projective modules.}
\end{enumerate}

\end{corollary}

\begin{proof}
For \eqref{item_PP} $\implies$ \eqref{item_Cg_filt}: suppose $M \in V$ is an $R$-module and $W$ is a $\sigma$-closed forcing extension of $V$ where $M$ becomes projective.  Then by Kaplansky's Theorem (in $W$), $M$ is $\mathcal{C}^W$-filtered, where $\mathcal{C}^W$ is the collection of countably-generated projective modules in $W$.  But every element of $\mathcal{C}^W$ is an element of $V$ by Lemma \ref{lem_CtblePresSigmaClosed}, and is projective in $V$ by Lemma \ref{lem_Proj_Abs}.  Hence, from the point of view of $W$, $\mathcal{C}^W  \subseteq \mathcal{C}^V$ (in fact they are equal), and so $W \models$ ``$M$ is $\mathcal{C}^V$-filtered".

For the converse, suppose $M$ is an $R$-module in $V$, and $W$ is a $\sigma$-closed forcing extension in which $M$ becomes $\mathcal{C}^V$-filtered.  Since projectivity is upward absolute, each module in $\mathcal{C}^V$ is projective in $W$, and hence $W$ sees that $M$ is a transfinite extension of projective modules, and is hence projective.
\end{proof}

An important fact about $\sigma$-closed forcing is:
\begin{fact}[Jech~\cite{MR1940513}]\label{fact_SigmaClosedPreserveStat}
If $\mathbb{P}$ is $\sigma$-closed, then for all uncountable $X$ in the ground model:
\begin{itemize}
 \item Closed unbounded subsets of $\wp_{\omega_1}(X)$ in the ground model remain closed unbounded in $V^{\mathbb{P}}$;\footnote{This is true for larger $\mu$ too, simply because $<\mu$-closed forcings add no new sets of size less than $\mu$.}
 \item Stationary subsets of $\wp_{\omega_1}(X)$ in the ground model remain stationary in $V^{\mathbb{P}}$.
\end{itemize}
\end{fact} 

The second bullet is specific to $\mu = \omega_1$; in general, $<\omega_2$-closed forcing can kill the stationarity of some subset of $\wp_{\omega_2}(X)$.  The next lemma is similar to the $\mu = \omega_1$ instance of Corollary \ref{cor_Aleph_1_gen_CFilt}.  The difference is that Corollary \ref{cor_Aleph_1_gen_CFilt} assumed $M$ was $\aleph_1$-generated, while in Corollary \ref{cor_StatLogicChar_SigmaPotCfilter}, there is no cardinality restriction on $M$.

\begin{corollary}\label{cor_StatLogicChar_SigmaPotCfilter}
Let $\mathcal{C}$ be a collection of countably presented modules, and $M$ be a module.  The following are equivalent:
\begin{enumerate}
 \item\label{item_M_SigmaPotCFiltered} $M$ is $\sigma$-closed potentially $\mathcal{C}$-filtered;
 
 \item In some $\sigma$-closed forcing extension, there is a $\mathcal{C}$-filtration of $M$ of (ordinal) length at most $\omega_1$; 
 \item\label{item_aa_aa_NoSizeRest} There is a closed unbounded $D \subseteq \wp_{\omega_1}(M)$ such that whenever $Z \subseteq Z'$ are both in $D$, $\langle Z' \rangle / \langle Z \rangle$ is $\mathcal{C}$-filtered, via a filtration of length less than $\omega_1$.
 \item\label{item_aa_stat_NoSizeRest}  there is a closed unbounded set $D \subseteq \wp_{\omega_1}(M)$ such that for every $Z \in D$, 
\begin{equation*}
\begin{split}
S_Z:= \{ Z' \in D \ : \ & Z \subset Z' \text{ and } \langle Z' \rangle / \langle Z \rangle \text{ is } \mathcal{C} \text{-filtered} \\ 
& \text{via a filtration of length $<\omega_1$} \}
\end{split}
\end{equation*}
is stationary in $\wp_{\omega_1}(M)$.

\end{enumerate}
\end{corollary}
\begin{proof}
This follows directly from Corollary \ref{cor_Aleph_1_gen_CFilt} and Fact \ref{fact_SigmaClosedPreserveStat}.

\end{proof}

\section{Filtration Games}\label{sec_FiltGames}

In Section \ref{sec_FiltGameDef} we introduce the notion of a $\mathcal{C}$-Filtration Game, played on a single module,  and prove some basic facts about these games, including their connection to potentially $\mathcal{C}$-filtered modules.  Section \ref{subsec_WS1AEC} shows that, in many cases, the class of modules for which Player 1 has a winning strategy in a Filtration Game constitutes an AEC.  Section \ref{subsec_DualBasisGame} introduces the Dual Basis Game on a module, which gives an alternate (and somewhat more concrete) description of the $\mathcal{C}$-Filtration Game in the case where $\mathcal{C}$ is the collection of countably generated, projective modules.  Section \ref{SubSec_DB_EF} shows that determinacy of Filtration Games implies determinacy of certain Ehrenfeucht-Fra\"iss\'{e} games.

\subsection{Filtration games and potential filtrations}\label{sec_FiltGameDef}

\begin{definition}\label{def_FiltGame}
Let $\mu$ be a regular (not necessarily uncountable) cardinal, $\mathcal{C}$ be a collection of $R$-modules, and $M$ an $R$-module.  The \textbf{$\boldsymbol{\mathcal{C}}$-filtration game of length $\boldsymbol{\mu}$ on $\boldsymbol{M}$}, denoted $\boldsymbol{\mathcal{G}_{\mu}^{\textbf{Filt}(\mathcal{C})}(M)}$, is a game between two players with at most $\mu$ innings, as follows:
\begin{center}
\begin{tabular}{|c|c|c|c|c|c|}
\hline 
Player 1 & $X_0$ & $X_1$ &  \dots & $X_i$ & \dots \\ 
\hline 
Player 2 & $Z_0$ & $Z_1$ &  \dots & $Z_i$ & \dots \\ 
\hline 
\end{tabular} 
\end{center}
At the beginning of round $i$ (where $i < \mu$), Player 1 plays an $X_i \subseteq M$ of size $<\mu$, and Player 2 responds with a $Z_i \subseteq M$ of size $<\mu$ such that 
\[
X_i \cup \bigcup_{k < i} Z_k \subseteq \langle Z_i \rangle
\]
and
\begin{equation}\label{eq_FilterReq}
\frac{\langle Z_i \rangle}{\left\langle \bigcup_{k < i } Z_k \right\rangle } \text{ is } \mathcal{C} \text{-filtered}. \tag{+}
\end{equation}

\noindent Player 2 wins if she lasts $\mu$ rounds; otherwise Player 1 wins.
\end{definition}

Note that for successor ordinals $i$, \eqref{eq_FilterReq} just means that $\langle Z_i \rangle / \langle Z_{i_0} \rangle$ is $\mathcal{C}$-filtered, where $i = i_0 + 1$.  If $\mathcal{C}$ is closed under transfinite extensions of length $<\mu$, then the game from Definition \ref{def_FiltGame} is the same as if we replaced requirement \eqref{eq_FilterReq} with the apparently stronger requirement that
\begin{equation}
\frac{\langle Z_i \rangle}{\left\langle \bigcup_{k < i } Z_k \right\rangle } \text{ is (isomorphic to an element) in }  \mathcal{C}. \tag{++}
\end{equation}

\begin{remark}\label{rem_ProjFiltGame}
When $\mathcal{C}$ is the collection of all $<\mu$-presented, projective modules, we may sometimes refer to the game $\mathcal{G}^{\text{Filt}(\mathcal{C})}_\mu(M)$ as the \textbf{Projective Filtration game of length $\boldsymbol{\mu}$ on $\boldsymbol{M}$}.  When $\mathcal{C} = \{ R \}$, we may sometimes refer to the game $\mathcal{G}^{\text{Filt}(\mathcal{C})}_\mu(M)$ as the \textbf{Free Filtration game of length $\boldsymbol{\mu}$ on $\boldsymbol{M}$}. 
\end{remark}

The following lemma dispenses with a technicality.  Recall that the definition of ``$\mathcal{C}$-filtered" requires successive factors in the filtration to be ``merely" \emph{isomorphic to} an element of $\mathcal{C}$, not necessarily \emph{in} $\mathcal{C}$; this is the convention, since it is convenient to sometimes work with $\mathcal{C}$'s that are not closed under isomorphism.  However, even if one did require $\mathcal{C}$ to be closed under isomorphism, one would still run into trouble again when going to a forcing extension of the universe, which may add new modules and cause $\mathcal{C}$ to no longer be closed under isomorphisms.  The following lemma says this isn't an issue in the situations we will encounter:

\begin{lemma}\label{lem_AbsCfilt}
Suppose $\mathcal{C}$ is a collection of $<\mu$-generated modules in the ground model $V$, $N \in V$ is a module, and $W$ is a $<\mu$-closed forcing extension of $V$.  Suppose $\vec{U} = \langle U_\xi \ : \ \xi < \zeta \rangle$ is, in $W$, a $\mathcal{C}$-filtration of $N$, and that $\zeta < \mu$.  Then $\vec{U}$ is an element of $V$, and $V \models$ ``$\vec{U}$ is a $\mathcal{C}$-filtration of $N$".
\end{lemma}
\begin{proof}
Let $\langle \dot{U}_\xi \ : \ \xi < \zeta \rangle$ be a name for the sequence.  Since $\mathcal{C}$ consists of $<\mu$-generated modules, and $\zeta < \mu$, it is forced that for all $\xi < \zeta$, $\dot{U}_\xi$  is a $<\mu$-generated submodule of $N$, and hence (by $<\mu$-closure of $V$ in $W$ and because $N \in V$) an element of $V$.  Again by $<\mu$ closure of $V$ in $W$, the entire sequence $\dot{\vec{U}}$ is forced to be in $V$.  Moreover, if $\xi +1 < \zeta$, then $W \models$ ``$U_{\xi+1}/U_\xi$ is isomorphic to an element of $\mathcal{C}$"; let $N_\xi \in \mathcal{C}$ witness this fact, and note that $N_\xi \in V$ because $\mathcal{C} \in V$.  Finally, since $N_\xi$ and $U_{\xi+1}/U_\xi$ are both in $V$, and are both $<\mu$-generated, any isomorphism between them in $W$ is already in $V$ (since it is determined by what it does to the generating set of the domain of size $<\mu$).  So $V \models$ ``$N_\xi$ is isomorphic to $U_{\xi+1}/U_\xi$".
\end{proof}

\begin{lemma}\label{lem_equiv_WS2Filter_Potent}
Let $\mu$ be a regular (not necessarily uncountable) cardinal, $\mathcal{C}$ a collection of $<\mu$-presented modules, and $M$ be a module.  The following are equivalent:
\begin{enumerate}
 \item $M$ is $<\mu$-closed potentially $\mathcal{C}$-filtered.
 \item Player 2 has a winning strategy in the game $\mathcal{G}^{\text{Filt}(\mathcal{C})}_{\mu}(M)$. 
\end{enumerate}
\end{lemma}
\begin{proof}
First suppose $M$ is $<\mu$-closed potentially $\mathcal{C}$-filtered.  By Lemma \ref{lem_SigmaClosedLengthOmega1}, there is a $<\mu$-closed forcing $\mathbb{P}$ that forces the existence of a $\mathcal{C}$-filtration of $M$ of length at most $\mu$. Let $\dot{\vec{Z}}=\Big\langle \langle \dot{Z}_i \rangle \ : \ i < \mu \Big\rangle$ be a name for this filtration, where each $\dot{Z}_i$ is forced to be a set of size $<\mu$.  By $<\mu$-closure of $\mathbb{P}$, every proper initial segment of $\dot{\vec{Z}}$ is forced to be in the ground model, and each $\langle \dot{Z}_{i+1} \rangle / \langle \dot{Z}_i \rangle$ is forced to be $\mathcal{C}$-filtered via a filtration of length $<\mu$.  By Lemma \ref{lem_AbsCfilt}, each such filtration is forced to be in the ground model, and moreover to be a $\mathcal{C}$-filtration there.  Then Player 2 can use the name $\dot{\vec{Z}}$ to obtain a winning strategy in $\mathcal{G}^{\text{Filt}(\mathcal{C})}_{\mu}(M)$, as follows:  suppose $\xi < \mu$ and the game so far looks like the following:

\begin{center}
\begin{tabular}{|c|c|c|c|c|c|c|}
\hline 
Player 1 & $X_0$ & $X_1$ &  \dots & $X_\zeta$ & \dots & $X_\xi$ \\ 
\hline 
Player 2 & $U_{0}$ & $U_{1}$ &  \dots & $U_{\zeta}$ & \dots & \\ 
\hline 
\end{tabular} 
\end{center}

\noindent Suppose also that Player 2 has created a descending sequence $\langle p_\zeta \ : \ \zeta < \xi \rangle$ of conditions in $\mathbb{P}$, and a strictly increasing sequence $\langle i_\zeta \ : \ \zeta < \xi \rangle$ of ordinals below $\mu$, such that for all $\zeta < \xi$,
\[
p_\zeta \Vdash \  \check{U}_\zeta = \dot{Z}_{i_\zeta} \text{ and } \langle \dot{Z}_{i_\zeta} \rangle \supseteq \check{X}_\zeta 
\]

\noindent By $<\mu$-closure of $\mathbb{P}$ there is a lower bound $p^*_\xi$ of the sequence $\langle p_\zeta \ : \ \zeta < \xi \rangle$.  Now $p^*_\xi$ forces ``there is some $i < \omega_1$ such that $\langle \dot{Z}_{i} \rangle \supseteq \check{X}_\xi$, $\dot{Z}_i$ is in the ground model, and $i > i_\zeta$ for all $\zeta < \check{\xi}$".   Pick some $p_\xi \le p^*_\xi$ that decides the value of such an $i$ and also decides a ground model set $Z$ for $\dot{Z}_i$. Player 2 then responds to $X_\xi$ with the set $Z$.  It is routine to check that this procedure yields a valid response at  any round before $\mu$, and hence Player 2 wins.

For the $\Leftarrow$ direction, suppose $\tau$ is a winning strategy for Player 2 in the game.  Define a poset $\mathbb{Q}$ where conditions are $<\mu$-length partial runs of the game where Player 2 has used $\tau$ along the way, ordered by end-extension.  Clearly $\mathbb{Q}$ is $<\mu$-closed.  Furthermore, an easy density argument ensures that for every $x \in M$ and every $q \in \mathbb{Q}$, there is a $q'$ extending $q$ such that $x$ is in one of the sets played by Player 1 at some time in the partial run $q'$.  It follows that if $G$ is generic for $\mathbb{Q}$ and $\vec{Z}_G$ is the union of Player 2's moves, $\vec{Z}_G$ is a $\mathcal{C}$-filtration of $M$ in $V[G]$.

\end{proof}

\begin{lemma}
Suppose $\mu$ is a regular cardinal, $\mathcal{C}$ is a collection of $<\mu$-presented modules, and $M$ is a $\mu$-generated module.  The following are equivalent:
\begin{enumerate}
 \item\label{item_ActuallyCF} $M$ is $\mathcal{C}$-filtered.

 \item\label{item_LessMuPCF} $M$ is $<\mu$-closed potentially $\mathcal{C}$-filtered.

 \item\label{item_WS_MuGen} Player 2 has a winning strategy in $\mathcal{G}^{\text{Filt}(\mathcal{C})}_\mu(M)$.
\end{enumerate}
\end{lemma}
\begin{proof}
Equivalence of \eqref{item_LessMuPCF} with \eqref{item_WS_MuGen} is true for all modules, by Lemma \ref{lem_equiv_WS2Filter_Potent}.  The implication \eqref{item_ActuallyCF} $\implies$ \eqref{item_LessMuPCF} is trivial.  Finally, to see that \eqref{item_WS_MuGen} implies \eqref{item_ActuallyCF}, fix a $\mu$-sized generating set $X$ for $M$, let $\tau$ be a winning strategy for Player 2 in the $\mathcal{C}$-filtration game of length $\mu$, and play a run of the game where Player 1 enumerates $X$ and Player 2 uses $\tau$.  Since $\tau$ is a winning strategy, this will yield a $\mathcal{C}$-filtration of $M$. 

\end{proof}

%The Gale-Stewart Theorem yields the following corollary, which could alternatively be proven via an absoluteness argument (via a tree ``searching for a dual basis" of the countably generated module):
%\begin{corollary}
%If $\mathcal{H} = (H,\in)$ and $\mathcal{H'} = (H',\in)$ are transitive models of $\text{ZFC}^-$ with $H \subseteq H'$, then for every ring $R \in H$ and every $R$-module $M \in H$ such that $\mathcal{H} \models$ ``$M$ is countably generated",  
%\[
%\mathcal{H} \models \ M \text{ is projective} \ \iff \ \mathcal{H}' \models M \text{ is projective.}
%\]
%\end{corollary}
%\begin{proof}
%\begin{com}
%to fill in
%\end{com}
%
%\begin{com}
%WAIT; THE PREVIOUS LEMMAS WERE ABOUT $<\mu$-presented, not $<\mu$-generated.  Go back and change to generated, where possible (make sure it doesn't call on the Hill Lemma stuff).
%\end{com}
%\end{proof}

\subsection{Winning Strategies for Player 1 and AECs}\label{subsec_WS1AEC}

Recall the definition of an Abstract Elementary Class (AEC), due to Shelah:
\begin{definition}\label{def_AEC}
Given a language $\mathcal{L}$ and a class $K$ of $\mathcal{L}$-structures, a partial order $\mathbb{K}=(K, \prec_K)$ is called an \textbf{Abstract Elementary Class (AEC)} iff:
\begin{enumerate}[label=(\roman*)]
 \item If $M \in K$ and $M \simeq N$, then $N \in K$;
 \item If $M \prec_K N$, $M \simeq M'$, and $N \simeq N'$, then $M' \prec_K N'$;
 \item If $M \prec_K N$, then $M$ is a substructure of $N$;
 \item If $M_0$, $M_1$, and $M_2$ are elements of $K$, $M_0$ is a substructure of $M_1$, $M_0 \prec_K M_2$, and $M_1 \prec_K M_2$, then $M_0 \prec_K M_1$.
 \item\label{item_TarskiVaughtAxiom} (Tarski-Vaught axioms) If $\alpha$ is an ordinal and $\langle M_i \ : \  i < \alpha \rangle$ is a $\subseteq$-continuous and increasing $\prec_K$-chain of elements of $K$, then:
\begin{enumerate}
 \item\label{item_TV_ClosedUnderUnions}    $M:=\bigcup_{i < \alpha}M_i$, with the obvious interpretations of $\mathcal{L}$-symbols, is an element of $K$;
 \item\label{item_TV_UnionAboveEach} $M_i \prec_K M$ for each $i < \alpha$; and
 \item\label{item_TV_minimal} if $N \in K$ is such that $M_i \prec_K N$ for all $i < \alpha$, then $M \prec_K N$.
 \end{enumerate}
 \item\label{item_LS_axiom} (L\"owenheim-Skolem number)  There exists a cardinal $LS(\mathbb{K})$ such that whenever $A \subseteq B$ and $B \in K$, there exists an $A' \in K$ such that $A \subseteq A' \prec_K B$ and 
 \[
 |A'| \le |A| + \text{LS}(\mathbb{K}).
 \]
\end{enumerate}
\end{definition}

Given a ring $R$, a regular infinite cardinal $\mu$, and a collection $\mathcal{C}$ of $<\mu$-presented $R$-modules, recall from the introduction that $\Gamma^{\text{Filt}(\mathcal{C})}_{\mu, \text{P1}}$ denotes the class of $R$-modules $M$ such that Player 1 has a winning strategy in the game $\mathcal{G}^{\text{Filt}(\mathcal{C})}_{\mu}(M)$.  And $\prec_R$ denotes the embeddability relation for $R$-modules.

%
%
%Given a regular cardinal $\mu$, a ring $R$, and a class $\mathcal{C}$ of $R$-modules, consider the class partial order 
%\begin{equation}\label{eq_PO_Player1}
%\Big( \Gamma  \big( \mathcal{G}^{\text{Filt}(\mathcal{C})}_{\mu} , \text{P1} \big), \le \Big)
%\end{equation}
%where:
%\begin{itemize}
% \item $\Gamma  \big( \mathcal{G}^{\text{Filt}(\mathcal{C})}_{\mu} , \text{P1} \big)$ is the class of $R$-modules $M$ such that Player 1 has a winning strategy in $\mathcal{G}^{\text{Filt}(\mathcal{C})}_{\mu}(M)$; and
% \item $\le$ is the ordinary submodule relation; more precisely, $M \le N$ if there is an injective $R$-module homomorphism from $M$ to $N$.  \end{itemize}

\begin{remark}
Everything we discuss in this section would also work if we had instead defined $M \prec_R N$ to mean that there is a \textbf{pure} embedding from $M$ to $N$, or even an elementary embedding in with respect to the language of $R$-modules.
\end{remark}

We show that under certain assumptions on $\mathcal{C}$, the class
\[
\Big(  \Gamma^{\text{Filt}(\mathcal{C})}_{\mu, \text{P1}}, \prec_R  \Big)
\]
is an AEC.  The only requirements from Definition \ref{def_AEC} that are not trivial to verify are the requirements \ref{item_TarskiVaughtAxiom} and \ref{item_LS_axiom}.

\begin{lemma}\label{lem_P1_LS_always}
For any infinite regular $\mu$, any ring $R$, and any class  $\mathcal{C}$ of $R$-modules, the partial order $\Big( \Gamma^{\text{Filt}(\mathcal{C})}_{\mu,\text{P1}} , \prec_R \Big)$ satisfies the L\"owenheim-Skolem axiom of Definition \ref{def_AEC}, as witnessed by the cardinal 
\[
\kappa:=\text{max} \ \Big(|R|,2^{<\mu}\Big).
\]
\end{lemma}

\begin{proof}
Suppose $M$ is an $R$-module and Player 1 has a winning strategy in the game $\mathcal{G}^{\text{Filt}(\mathcal{C})}_{\mu}(M)$; let $\tau$ be such a strategy.  Let $A$ be a subset of $M$ of cardinality at most $\kappa$.  We need to find a submodule of $M$ of size at most $\kappa$, containing $A$, for which Player 1 still has a winning strategy.  Fix a regular cardinal $\theta$ such that $M,R,\tau,A \in H_\theta$.  Since $\kappa^{<\mu} = \kappa$, there exists an
\[
X \prec (H_\theta,\in, M,R, \tau, A, \mathcal{C} \cap H_\theta)
\]  
such that $|X|=\kappa \subset X$ and $X$ is closed under $<\mu$-length sequences.  Since $R \in X$ and $|R|\le \kappa \subset X$, it follows that $R \subset X$, and hence that $X \cap M$ is closed under multiplication from $R$.  Hence $X \cap M$ is an $R$-submodule of $M$; in fact, it is an elementary submodule in the language of $R$-modules.   Similarly, since $A \in X$ and $|A|\le \kappa \subset X$, it follows that $A \subset X$.  So $X \cap M \supseteq A$.

We claim that Player 1 has a winning strategy in $\mathcal{G}^{\text{Filt}(\mathcal{C})}_{\mu}\big( X \cap M \big)$.  This follows from the fact that $\tau \in X$, $X$ is closed under $<\mu$- sequences, and the definition of the game requires both players to play $<\mu$-sized subsets of $M$ at each stage.  More precisely, the winning strategy is just the restriction of $\tau$ to partial, $<\mu$-length runs of the game $\mathcal{G}^{\text{Filt}(\mathcal{C})}_{\mu}(M)$ such that Player 2's moves are always contained in $X \cap M$.  The model $X$ is closed under such plays, since $\tau \in X$ and $X$ is $<\mu$-closed.
\end{proof}

\begin{definition}\label{def_QuotHered}
Let us say that a class $\mathcal{C}$ of $R$-modules is \textbf{quotient-hereditary} if the following holds:    whenever $Z \subseteq Z'$ are $R$-modules such that 
\[
Z \text{ and } Z'/Z \text{ are } \mathcal{C} \text{-filtered,}
\]
then for all submodules $U'$ of $Z'$,
\[
U' \cap Z \text{ and } (U' \cap Z')/(U' \cap Z) \text{ are } \mathcal{C} \text{-filtered.}
\]
\end{definition}

%\begin{com}
%EDIT 9/19:  how is this related to the statement ``$\text{Filt}(\mathcal{C})$ is closed under submodules" (i.e.\ ``$\text{Filt}(\mathcal{C})$ is hereditary class")?  I suspect quotient-hereditariness is stronger.  We seem to need it below; if you merely assumed that $\text{Filt}(\mathcal{C})$ is hereditary, it is not clear that the filtration for the big module would ``cut down" in a nice way for quotients.
%
%Probably quotient-hereditariness of $\mathcal{C}$ should imply that $\text{Filt}(\mathcal{C})$ is hereditary however; check it and if true, mention it.
%
%Let's check it:  suppose $\mathcal{C}$ is quotient hereditary, and $M$ is $\mathcal{C}$-filtered by $\vec{M} = \langle M_i \ : \ i < \eta \rangle$.  Let $N$ be a submodule of $M$.  Then
%\[
%\langle N \cap M_i \ : \ i < \eta \rangle
%\]
%is a filtration of $N$.  
%
%Well, I think that \textbf{we need to assume $\mathcal{C}$ is closed under transfinite extensions (to conclude that quotient hereditaryness of $\mathcal{C}$ implies ordinary hereditariness of the class $\text{Filt}(\mathcal{C})$).  \color{blue} ON THE OTHER HAND, if you weakened the notion of quotient hereditary to conclude only that $(U' \cap Z')/(U' \cap Z)$ is $\mathcal{C}$-filtered \dots \color{red}}  Assume that from now on.  This implies that each $M_i$ is in $\mathcal{C}$ (not just adjacent quotients).  It follows that $N \cap M_i$ and $\frac{N \cap M_{i+1}}{N \cap M_i}$
%\end{com}

\begin{lemma}\label{lem_UpClosed}
Suppose $\mathcal{C}$ is a quotient-hereditary class of $R$-modules. Then for any infinite regular $\mu$, $\Gamma^{\text{Filt}(\mathcal{C})}_{\mu,\text{P1}}$ is upward closed under the submodule relation.
\end{lemma}
\begin{proof}
Assume $M$ is a submodule of $N$ and Player 1 has a winning strategy $\tau$ in the game on $M$.  Define a strategy $\tau^N$ for Player 1 in the game on $N$ as follows:
\begin{itemize}
 \item $\tau^N(\emptyset):= \tau(\emptyset)$; note this is a $<\mu$-sized subset of $M$.
 \item Suppose $j < \mu$ and 
\[
p = \big \langle (X_i,Z_i) \ : \ i < j \big \rangle
\]
is a partial run of the game on $N$ of length $j$, where the $X_i$'s are Player 1's moves and the $Z_i$'s are Player 2's responses.
\begin{itemize}
 \item If at least one $X_i$ fails to be contained in $M$, then $\tau^N(p)$ is undefined.
 \item Otherwise, let
\[
p|M:= \big \langle (X_i, Z_i \cap M) \ : \ i < j \big \rangle
\]
Since $p$ is a partial run of the game, for each $i < j$ the modules $\bigcup_{k < i} \langle Z_k \rangle$ and $\langle Z_i \rangle / \bigcup_{k < i} \langle Z_k \rangle$ are
 $\mathcal{C}$-filtered; and since $\mathcal{C}$ is quotient-hereditary, it follows that 
\[
M \cap \bigcup_{k < i } \langle Z_k \rangle \text{ and } \frac{M \cap \langle Z_i \rangle}{M \cap \bigcup_{k < i } \langle Z_k \rangle}
\]
are $\mathcal{C}$-filtered.  Furthermore, since the $X_i$'s are assumed to be contained in $M$ by our case, and $Z_i \supset X_i$, it follows that each $M \cap \langle Z_i  \rangle = \langle Z_i \cap M \rangle$ contains $X_i$.  It follows that $p|M$ is a run of the game on $M$.  If $p|M$  is a game according to $\tau$, then set 
\[
\tau^N(p):= \tau(p|M).
\]
(Otherwise leave $\tau^N(p)$ undefined).
\end{itemize}

\end{itemize}

It follows that $\tau^N$ is a winning strategy for Player 1 in the game on $N$.  If not, there would be a game 
\[
\big\langle X_i, Z_i  \ : \ i < \mu \big\rangle
\]
 on $N$ lasting $\mu$ rounds, where Player 1 used $\tau^N$ along the way (in particular, $\tau^N$ was defined at every stage).  But then 
\[
\big\langle X_i, Z_i \cap M \ : \ i < \mu \big\rangle
\]
is a game of length $\mu$ on $M$, where Player 1 used $\tau$.  This contradicts that $\tau$ was a winning strategy for Player 1 in the game on $M$.
\end{proof}

We can now complete the proof of Theorem \ref{thm_AlwaysAEC} (except the ``in particular" part, which is taken care of byLemma \ref{lem_HeredRingQuotHered} below).  Lemma \ref{lem_UpClosed} immediately implies clause \eqref{item_TV_ClosedUnderUnions} of the Tarski-Vaught axiom of Definition \ref{def_AEC}.  The fact that our ordering $\prec_R$ is just the submodule (or pure submodule, or elementary submodule) relation ensures that parts \eqref{item_TV_UnionAboveEach} and \eqref{item_TV_minimal} of the Tarski-Vaught axiom are straightforward.  Lemma \ref{lem_P1_LS_always} ensures that the L\"owenheim-Skolem axiom \ref{item_LS_axiom} of Definition \ref{def_AEC} holds.  The other requirements in Definition \ref{def_AEC} are routine.

Recall that a ring $R$ is \textbf{hereditary} if submodules of projective $R$-modules are always projective.  We next justify the various ``in particular" clauses that appear in the statements of the theorems in the introduction.

\begin{lemma}\label{lem_HeredRingQuotHered}
If $R$ is a hereditary ring, then the class of countably generated, projective $R$-modules is quotient-hereditary.  The same is true for the class of countably presented, projective $R$-modules.
\end{lemma}
\begin{proof}

By Kaplansky's Theorem, being filtered by countably generated (or countably presented) projective modules is equivalent to being projective.\footnote{Kaplansky's Theorem usually refers to the ``projective implies filtered by countably generated projectives" direction;  the other direction, i.e.\ that projectives are closed under transfinite extensions, is a special case of Eklof's Lemma~\cite{MR453520}.}  So we just need to show that, if $Z \subset Z'$ are $R$-modules such that $Z$ and $Z'/Z$ are projective, then for every submodule $U'$ of $Z'$, both $U' \cap Z$ and $U'/(U' \cap Z)$ are projective.

Now since $Z$ and $Z'/Z$ are projective, so is $Z'$, and hence (by hereditary property of $R$) $U'$ and $U' \cap Z$ are projective.  But projectivity of their quotient seems to require a little argument, which is a slight modification of an argument of Kaplansky (see page 42 of Lam~\cite{MR1653294}).  

Since $Z$ and $Z'/Z$ are projective, then $Z' \simeq Z \oplus Z'/Z$; by amalgamating dual bases for $Z$ and $Z'/Z$ in the obvious way, there exists a dual basis $\mathcal{D}' = \Big( B', (f'_b)_{b \in B'} \Big)$ for $Z'$ such that 
\begin{equation}\label{eq_ConsExt}
\forall x \in Z \ \text{sprt}_{\mathcal{D}'}(x):= \{ b \in B' \ : \ f'_b(x) \ne 0  \} \text{ is contained in } B' \cap Z.
\end{equation}

Fix an enumeration $\langle b_\xi \ : \ \xi < |B'| \rangle$ of $B'$ such that members of $B' \cap Z$ are enumerated first; say $\zeta_Z \le |B'|$ is such that
\[
B' \cap Z = \{ b_\xi \ : \ \xi < \zeta_Z \}
\]
Let $f'_\xi$ denote $f'_{b_\xi}$ for each $\xi < |B'|$.  For each $\mu \le |B'|$ let $\boldsymbol{\textbf{span}\Big( \mathcal{D}' \restriction \mu  \Big)}$ denote the set of $x \in Z'$ such that   
\[
\text{sprt}_{\mathcal{D}'}(x) \subseteq \{ b_\xi \ : \ \xi < \mu  \}.
\]
For each $\xi < |B'|$ let $U'_\xi:= U' \cap \text{span}\Big( \mathcal{D}' \restriction \xi \Big)$.  Note that the sequence 
\[
\Big\langle U'_\xi \ : \ \xi < |B'| \Big\rangle
\]
is a filtration of $U'$, and the union of the first $\zeta_Z$ many entries is $U' \cap Z$ by \eqref{eq_ConsExt}.  For each $\xi < |B'|$, define $J_\xi$ to be the image of the $R$-module homomorphism
\[
 f'_\xi \restriction U'_{\xi+1}   
\]
and observe that the kernel of this map is exactly $U'_\xi$.  Hence,
\begin{equation}\label{eq_Quotient_Iso_J}
\frac{U'_{\xi+1}}{U'_\xi} \ \simeq \ J_\xi.
\end{equation}
Now $J_\xi$ is an $R$-submodule of the (free) module $R$ (i.e., $J_\xi$ is an ideal in $R$), and hence by the assumption that $R$ is hereditary, \eqref{eq_Quotient_Iso_J} is projective for each $\xi < |B'|$.  So
\[
\langle U'_\xi \ : \ \xi < \zeta_Z \rangle \text{ is a projective filtration of } U' \cap Z,
\]  
and
\[
\langle U'_\xi \ : \ \xi < |B'| \rangle \text{ is a projective filtration of } U'.
\]
It follows from splitting properties of projective modules that
\[
U' \cap Z \ \simeq \ \bigoplus_{\xi < \zeta_Z} U'_{\xi+1}/U'_\xi \ \simeq \  \bigoplus_{\xi < \zeta_Z} J_\xi,
\]
\[  \ \ 
U'  \ \simeq \bigoplus_{\xi < |B'|} U'_{\xi+1}/U'_\xi \ \simeq \  \bigoplus_{\xi < |B'|} J_\xi ,
\]
and
\[
U'/(U' \cap Z) \simeq \bigoplus_{\xi \in \big[ \zeta_Z, \  |B'| \big)} J_\xi.
\]
Since each $J_\xi$ is projective, and direct sums of projective modules are projective, this completes the proof.
\end{proof}

\begin{corollary}\label{cor_AlwaysAEC}
If $R$ is a hereditary ring, then the class of $R$-modules for which Player 1 has a winning strategy in the Projective Filtration game of length $\omega_1$ (see Remark \ref{rem_ProjFiltGame}), under the embeddability (or pure embeddability) ordering, is an AEC.

The same statement holds if ``Projective Filtration Game" is replaced by ``Free Filtration Game".
\end{corollary}

  In Section \ref{subsec_DualBasisGame} we give an alternate, more concrete description of the Projective Filtration Game.

\subsection{The Dual Basis Game}\label{subsec_DualBasisGame}

This section is not essential to the rest of the paper, but we describe a slightly more constructive/intuitive game that is equivalent to the Projective Filtration Game; in particular, each move in this equivalent game involves just finitely many decisions.  We focus on $\mu = \omega_1$ for concreteness; i.e., we give an alternate description of the game $\mathcal{G}^{\text{Filt}(\mathcal{C})}_{\omega_1}(-)$, where $\mathcal{C}$ is the collection of countably presented, projective modules.  

Recall the definition of \emph{dual basis} from Section \ref{sec_Prelims}, and the fact that a module is projective if and only if it has a dual basis.  We introduce weaker variant of a dual basis.  Given an $R$-module $M$ and a subset $X$ of $M$, say that a pair $\mathcal{D} = \Big( B, \big( f_b \big)_{b \in B} \Big)$ is an \textbf{$\boldsymbol{M}$-Dual Basis for $\boldsymbol{\langle X \rangle}$} if it satisfies the usual requirements of being a dual basis for $\langle X \rangle$, except that we do \emph{not} require $B \subseteq \langle X \rangle$; merely that $B  \subseteq M$ (hence the $M$-prefix in the notation).  More precisely, $\mathcal{D}$ is an $M$-Dual Basis for $\langle X \rangle$ iff:
\begin{enumerate}
 \item $B \subseteq M$;
 \item Each $f_b$ is an $R$-linear map from $\langle X \rangle \to R$;
 \item For each $x \in  X $, $\text{sprt}_{\mathcal{D}}(x):= \{ b \in B \ : \ f_b(x) \ne 0 \}$ is finite, and $x = \sum_{b \in \text{sprt}_{\mathcal{D}}(x)} f_b(x) b$.\footnote{It is an easy exercise to verify that if this requirement holds of every element of $X$, it also holds for every element of $\langle X \rangle$. }
\end{enumerate}

\noindent If $X \subseteq Y \subseteq M$, $\mathcal{D}^X = \Big(B^X, \big( f^X_b \big)_{b \in B_X} \Big)$ is an $M$-Dual Basis for $\langle X \rangle$, and $\mathcal{D}^Y= \Big(B^Y, \big( f^Y_b \big)_{b \in B_Y} \Big)$ is an $M$-Dual Basis for $\langle Y \rangle$, we say that $\mathcal{D}^Y$ is a \textbf{conservative extension} of $\mathcal{D}^X$ if $B^Y \supseteq B^X$, $f^Y_b$ extends $f^X_b$ for all $b \in B^X$, and for every $x \in X$, $\text{sprt}_{\mathcal{D}^X}(x) = \text{sprt}_{\mathcal{D}^Y}(x)$.

\begin{definition}
Let $R$ be a ring and $M$ an $R$-module.  The \textbf{Dual Basis Game on $\boldsymbol{M}$ of length $\boldsymbol{\omega_1}$}, denoted $\boldsymbol{\textbf{DB}_{\omega_1}(M)}$, is the two-player game where, at each countable round $i$:
\begin{itemize}
 \item Player 1 plays an element $x_i$ of $M$.
 \item Player 2 responds, if possible, with an $M$-Dual Basis 
 \[
 \mathcal{D}^i = \Big( B^i,  \big( f^i_b \big)_{b \in B^i} \Big)
 \]
 for $\Big\langle  \{ x_k \ : \ k \le i \} \cup \bigcup_{k < i} B^k \Big\rangle$ such that:
 \begin{enumerate}
  \item $\mathcal{D}^i$ conservatively extends $\mathcal{D}^k$ for all $k < i$; 
  \item  $B^i \setminus \bigcup_{k < i } B^k$ is finite.
 \end{enumerate}
 
\end{itemize}

%\begin{center}
%\begin{tabular}{|c|c|c|c|c|c|}
%\hline 
%Player 1 & $x_0$ & $x_1$ &  \dots & $x_i$ & \dots  \\ 
%\hline 
%Player 2 & $\mathcal{D}_{0}$ & $\mathcal{D}_{1}$ &  \dots & $\mathcal{D}_{i}$ & \dots \\ 
%\hline 
%\end{tabular} 
%\end{center}

Player 2 wins if she lasts $\omega_1$ rounds; otherwise Player 1 wins. 
\end{definition}

\noindent The conservativity requirement for the extensions ensures that supports don't grow into an infinite set as the game progresses.  

The game $\text{DB}_{\omega_1}(M)$ is equivalent to the Projective Filtration Game of length $\omega_1$ on $M$ introduced in Remark \ref{rem_ProjFiltGame}.  Since we will not use this fact, but only mention it as a curiosity, we will omit the proof.  The basic reason the games are equivalent is:

\begin{fact}\label{fact_ProjQuot}
For any modules $Z \subset Z'$, the following are equivalent:
\begin{enumerate}
 \item $Z$ and $Z'/Z$ are projective;
 \item There is a dual basis $\mathcal{D}$ for $Z$ that can be \textbf{conservatively extended} to a dual basis $\mathcal{D}'$ for $Z'$; this means that $\mathcal{D}'$ extends $\mathcal{D}$ in the obvious way and, moreover, for all $z \in Z$, $\text{sprt}_{\mathcal{D}}(z) = \text{sprt}_{\mathcal{D}'}(z)$.  (Note: here we are referring to ordinary dual bases, not $M$-dual bases).
 \item\label{item_EveryDBextendsCons} $Z$ is projective and every dual basis for $Z$ can be conservatively extended to a dual basis for $Z'$.
\end{enumerate}  
\end{fact}

Although they are equivalent, the games $\text{DB}_{\omega_1}(M)$ and the Projective Filtration Game of length $\omega_1$ on $M$ proceed at different paces; the former involves finitely many new objects at each round, while the latter involves countably many.

%This lag is overcome by the fact that the definition of the Dual Basis Game only requires Player 2 to play an $\boldsymbol{M}$-Dual Basis for $\langle  \{ x_k \ : \ k \le i \} \rangle$ at round $i$, rather than an ordinary dual basis; i.e., her $B^i$ component of $\mathcal{D}^i$ is allowed to include objects outside of $\langle \{ x_k \ : \ k \le i \} \rangle$.  If we removed the requirement that $B_i \setminus \bigcup_{k < i} B_k$ be finite, then one could instead require her to play an ordinary dual basis for $\langle \{ x_k \ : \ k \le i \} \rangle$.  This game would also be equivalent to the Projective Filtration Game.
%
%is the reason that Player 2 only plays $M$-Dual Bases, rather than ordinary Dual Bases, at each stage of the Dual Basis Game.  
%
%%Given a challenge by Player 1---in either the Dual Basis Game on $M$, or the Projective Filtration Game on $M$---Player 2 can either view her job as:
%%\begin{itemize}
%% \item  finding some countably presented $Z_i$ containing Player 1's new challenge such that $Z_i / \bigcup_{k < i } Z_k$ is projective (in the Projective Filtration Game); or
%% \item finding an $M$-Dual Basis conservatively extending  conservative extension of a dual basis for $\bigcup_{k < i} Z_k$ that contains Player 1's new challenge(s) in its dual basis span (in the Dual Basis Game).  
%%\end{itemize}
%%\noindent Fact \ref{fact_ProjQuot} ensures these are essentially the same problem for her.  Since the Dual Basis Game involves only finitely many new plays at each stage, there is some ``lag" between the games, but this does not affect their equivalence.
%%
%

\subsection{Relation to Ehrenfeucht-Fra\"iss\'{e} games}\label{SubSec_DB_EF}

In this section we show that determinacy of Filtration Games implies determinacy of  certain Ehrenfeucht-Fra\"iss\'{e} games, including the kind appearing in Mekler et al.~\cite{MR1191613}.

If $\mathfrak{A}=(A,\dots)$, $\mathfrak{B}=(B,\dots)$ is a pair of structures in the same (relational) signature, and $\mu$ is an ordinal, the \textbf{Ehrenfeucht-Fra\"iss\'{e} game of length $\mu$ on the pair $\mathfrak{A},\mathfrak{B}$}, denoted $\textbf{EF}_{\boldsymbol{\mu}}(\boldsymbol{\mathfrak{A},\mathfrak{B}})$, is the game lasting (at most) $\mu$ innings, where at the top of inning $i$, \emph{Spoiler} plays an element of $A \cup B$, and \emph{Duplicator} then responds with an element of the other structure, in such a way that the pairs of elements chosen so far constitutes a partial isomorphism between $\mathfrak{A}$ and $\mathfrak{B}$.  \emph{Duplicator} wins if she lasts $\mu$ innings; otherwise \emph{Spoiler} wins.  

By the Gale-Stewart Theorem, $\text{EF}_\omega(-,-)$ is always determined.  Morever, this is related to ``potential isomorphisms":  Player 2 has a winning strategy in $\text{EF}_{\omega}(\mathfrak{A},\mathfrak{B})$ if and only if $\mathfrak{A}$ is isomorphic to $\mathfrak{B}$ in some forcing extension (see Nadel-Stavi~\cite{MR462942}).

Determinacy of $\text{EF}_{\omega}(-,-)$ does not generalize to higher cardinals:  Mekler et al.~\cite{MR1191613} give a ZFC example of a non-determined $\text{EF}_{\omega_1}(-,-)$ game.  On the other hand, Mekler et al.~\cite{MR1191613} did prove that it is relatively consistent with ZFC plus large cardinals that $\text{EF}_{\omega_1}(G,F)$ is determined whenever $G$ is a group and $F$ is a free abelian group.  Our Theorem \ref{thm_Cox_MM_Det} can be viewed as a strengthening of their theorem, in light of the results below.

%\begin{com}
%I don't think you need to do it this way.  Rather, you allow uncountable languages, so it's OK to have a function $f_r(x):= rx$ for every $r \in R$.
%\end{com}
%If $R$ is a ring and $M$ and $N$ are $R$-modules each of cardinality $\ge |R|$, then one can code both the module operations using the underlying sets of $M$ and $N$.  But if $|R| > |M|$ or $|R|>|N|$, the definition above is somewhat unnatural since, to code the multiplication, one would need to enlarge the universe of either $M$ or $N$ (or both).   So if $M,N$ is a pair of $R$-modules, $\textbf{EF}^{\textbf{R}}_{\boldsymbol{\lambda}}\textbf{(M,N)}$ is defined as above, except that we also require of Player 2 that the addition-preserving partial map $\sigma_i: M \to N$ produced after her turn at stage $i$ also preserves multiplication by $R$, whenever this is defined.  In other words, for all $r \in R$ and all $x \in \text{dom}(\sigma_i)$, if $rx \in \text{dom}(\sigma_i)$, then 
%\[
%\sigma_i(rx) = r \sigma_i(x).
%\]
%We will often omit the superscript $R$ from $\text{EF}^R_{\lambda}(M,N)$ when it is clear from the context.

If $M$ and $N$ are $R$-modules, for the purposes of the game $\text{EF}_\mu(M,N)$ we will view them as structures in the $(|R|+\aleph_0)$-sized language that includes predicates describing the scalar multiplication by $R$ on the modules.  Then if $X \subset M$, $Y \subset N$, and $\sigma: X \to Y$ is a function, $\sigma$ is a partial isomorphism from $M \to N$ if and only if $\sigma$ lifts to an $R$-module isomorphism 
\[
\widehat{\sigma}: \langle X \rangle^M_R \to \langle Y \rangle^N_R.
\]

\begin{theorem}\label{thm_P1WSFG_implies_SplrWS}
Suppose $\mu$ is a regular uncountable cardinal, $R$ is a ring, $M$ is an $R$-module, and $\mathcal{C}$ is a class of $<\mu$-presented $R$-modules.  If Player 1 has a winning strategy in the game $\mathcal{G}^{\text{Filt}(\mathcal{C})}_{\mu}(M)$, then for every $\mathcal{C}$-filtered module $P$, \emph{Spoiler} has a winning strategy in $\text{EF}^R_{\mu}(M,P)$.
\end{theorem}
\begin{proof}
Let $\tau$ be a winning strategy for Player 1 in $\mathcal{G}^{\text{Filt}(\mathcal{C})}_{\mu}(M)$, and suppose $P$ is a $\mathcal{C}$-filtered module.  We use $\tau$ to define a winning strategy for \emph{Spoiler} in $\text{EF}_{\mu}(M,P)$.  This relies heavily on Lemma \ref{lem_HillClub}, which guarantees that there is a club $D \subseteq \wp_\mu(P)$ such that 
\[
\forall Z \subset Z' \text{ both in } D, \ \langle Z \rangle \text{ and }  \frac{\langle Z' \rangle}{\langle Z \rangle} \text{ are } \mathcal{C} \text{-filtered.}
\]

We now describe a strategy for \emph{Spoiler} in the game $\text{EF}_{\mu}(M,P)$, which we will denote by $\psi$.

\begin{itemize}
 \item Let $X_0:= \tau(\emptyset)$ be Player 1's opening move according to $\tau$ in the $\mathcal{C}$-filtration game on $M$.  \emph{Spoiler} begins the $\text{EF}_{\mu}(M,P)$ game by simply enumerating $X_0$ in his first $|X_0|$-many moves.  Let $P_0$ be the set of \emph{Duplicator}'s responses.  Fix a $Z_0 \in D$ such that $P_0 \subseteq Z_0$.  \emph{Spoiler} then uses his next $|Z_0|$-many moves to enumerate $Z_0$.  Assuming \emph{Duplicator} has survived so far (i.e., through these first $|X_0|+|Z_0|$ stages of the EF game), she has constructed a bijection
 \[
 \sigma_0: Y_0  \to  Z_0 
 \]
 where $Y_0$ is $<\mu$-sized subset of $M$ containing $X_0$, $\langle Z_0 \rangle^P_R$ is $\mathcal{C}$-filtered (because $Z_0 \in D$), and $\sigma_0$ lifts to an isomorphism 
 \[
 \widetilde{\sigma}_0: \langle Y_0 \rangle^M_R \to \langle Z_0 \rangle^P_R.
 \]
 \item Now go back to the Filtration Game, where it is the bottom of the 0-th inning, and make Player 2 play the set $Y_0$; this is a valid move because $\langle Y_0 \rangle$ is isomorphic to $\langle Z_0 \rangle$ and hence $\mathcal{C}$-filtered, and $X_0 \subset Y_0$.
% \item  Let $X_1$ be the response of Player 1 according to $\tau$.  Moving back to the EF game, \emph{Spoiler} enumerates the countable set $X_1$ in his next $\omega$ moves.  Let $P_1$ be the set of responses by \emph{Duplicator}.  Fix a $Z_1 \in D$ such that $Z_0 \cup P_1 \subseteq Z_1$.  In the next $\omega$ moves, \emph{Spoiler} enumerates $Z_1$.  Assuming \emph{Duplicator} has survived so far (so far there have been $\omega \cdot 4$ rounds in the EF game), she has created a bijection
% \[
% \sigma_1: Y_1 \to Z_1
% \] 
% for some countable $Y_1 \supseteq X_1$ such that $\langle Z_1 \rangle^P_R$ is $\mathcal{C}$-filtered, $\sigma_1 \supseteq \sigma_0$, and $\sigma_1$ lifts to an isomorphism
% \[
% \widetilde{\sigma}_1: \langle Y_1 \rangle^M_R \to \langle Z_1 \rangle^P_R.
% \] 
% \item Now go back to the Filtration Game, where it is Player 2's turn to move in the bottom of the 1st inning; have her play the set $Y_1$.  This is a valid response---in particular $\langle Y_1 \rangle / \langle Y_0 \rangle$ is $\mathcal{C}$-filtered---because $\widetilde{\sigma}_1$ extends $\widetilde{\sigma}_0$, and $\langle Z_1 \rangle / \langle Z_0 \rangle$ is $\mathcal{C}$-filtered. 

 \item In general, suppose we are at a position $t$ in the game tree of $\text{EF}_{\omega_1}(M,P)$, and that it is \emph{Spoiler}'s turn to play.  
 
%\begin{com}
%Well this won't be accurate for ALL $t$ where it's spoiler's turn to play; what if $t$ is ``in the middle" of one of his blocks of moves you describe below?  In other words, your things indexed by $t$ won't necessarily be indexed for ALL $t$ where it is currently Spoiler's move!
%\end{com} 
 
Suppose we have constructed:
 \begin{enumerate}
%  \item A strictly increasing sequence $\langle \alpha_s \ : \ s \vartriangleleft t \rangle$ of ordinals, indexed by initial segments $s$ of $t$, such that each $\alpha_s < \mu$;
  \item\label{item_ZsCfiltered} a $\subseteq$-increasing (but not necessarily continuous) sequence 
  \[
  \langle Z_i \ : \ i < \alpha_t \rangle
  \]
   of elements of $D$, for some $\alpha_t < \mu$, such that for all $i < \alpha_t$, 
  \[
  \frac{\langle Z_i \rangle}{\bigcup_{k < i} \langle Z_k \rangle} \text{ is } \mathcal{C} \text{-filtered};
  \]

  \item a partial play
  \[
p(t)=  \big\langle (X_i,Y_i) \ : \  i < \alpha_t \big\rangle
  \]
of the game $\mathcal{G}^{\text{Filt}(\mathcal{C})}_{\mu}(M)$ according to $\tau$;

  \item a coherent sequence of bijections $\langle \sigma_i: Y_i \to Z_i \ : \ i < \alpha_t \rangle$ arising from the plays so far of the EF game, such that each $\sigma_i$ lifts to an isomorphism
  \[
  \widetilde{\sigma}_i: \langle Y_i \rangle^M \to \langle Z_i \rangle^P.
  \]
%  \item $Z_i \supseteq ??? $ for all $i < \alpha$;
 \end{enumerate}
\textbf{Then} let $X_{\alpha_t}:=\tau(p(t))$.   In the next $|X_{\alpha_t}|$-many rounds of the EF game, \emph{Spoiler} enumerates $X_{\alpha_t}$.  Letting $P_{\alpha_t}$ be \emph{Duplicator}'s responses, let $Z_{\alpha_t}$ be an element of $D$ such that
\[
P_{\alpha_t} \cup \bigcup_{i < \alpha_t} Z_i \ \subseteq Z_{\alpha_t}.
\]
\noindent Observe that by closure of $D$, $\bigcup_{i < \alpha_t} Z_i \in D$, and hence
\[
\frac{\langle Z_{\alpha_t} \rangle }{\bigcup_{i < \alpha_t} \langle Z_i \rangle} \text{ is } \mathcal{C} \text{-filtered.}
\]
(So \eqref{item_ZsCfiltered} has been maintained).  \emph{Spoiler} then enumerates $Z_{\alpha_t}$ at his turn in the next $|Z_{\alpha_t}|$ rounds of the EF game.  Assuming \emph{Duplicator} has survived, the play of the EF game so far which is $|X_{\alpha_t}|+|Z_{\alpha_t}|$ (in ordinal arithmetic) many rounds beyond node $t$, yields some $<\mu$-sized $Y_{\alpha_t} \subseteq M$ containing $X_{\alpha_t}$, and a bijection
\[
\sigma_{\alpha_t}: Y_{\alpha_t} \to Z_{\alpha_t}
\] 
 extending $\sigma_i$ for all $i < \alpha_t$ that lifts to an isomorphism $\widehat{\sigma}_{\alpha_t}: \langle Y_{\alpha_t} \rangle^M \to \langle Z_{\alpha_t} \rangle^N$, and such that
\[
\frac{\langle Z_{\alpha_t} \rangle}{\bigcup_{k < \alpha_t} \langle Z_k \rangle} \text{ is } \mathcal{C} \text{-filtered.}
\]
Since $\widetilde{\sigma}_{\alpha_t}$ is an isomorphism extending the previous $\widetilde{\sigma}_i$'s, it follows that 
\[
\frac{\langle Y_{\alpha_t} \rangle }{\bigcup_{i < \alpha_t} \langle Y_i \rangle} \text{ is } \mathcal{C} \text{-filtered,}
\]
and hence that $Y_{\alpha_t}$ is an acceptable response to $p(t)^\frown X_{\alpha_t}$ by Player 2 in the Filtration game.

 \end{itemize}

We claim that $\psi$ is a winning strategy for \emph{Spoiler} in the EF game.  Suppose not; then by regularity of $\mu$ there exists a strictly increasing sequence
\[
\langle t_\xi \ : \ \xi < \mu \rangle
\]
of nodes in the EF game tree where \emph{Spoiler} has used the strategy $\psi$.  Then 
\[
\langle p(t_\xi) \ : \ \xi < \mu \rangle
\]
is a strictly increasing sequence of nodes in the Filtration Game where Player 1 has used the strategy $\tau$; this is a contradiction, because $\tau$ is a winning strategy for Player 1 in the Filtration Game.

% 
%The map 
%\[
%t \mapsto p(t)
%\] 
%is an order preserving map from
%\begin{com}
%NOT ALL OF THEM!!!
%\end{com}
%
% those nodes of the EF game tree that have been played according to $\psi$, into those nodes of the Filtration Game tree that have been played according to $\tau$.  Since the latter has no $\mu$-length branch (because $\tau$ is a winning strategy for Player 1 in the Filtration Game), no EF game according to $\psi$ can last $\mu$ many rounds.  So $\psi$ is a winning strategy for Player 1 in the Filtration Game.
%
%%Let $T_\psi$ denote those nodes in the game tree of $\text{EF}_{\omega_1}(M,N)$ that are according to $\psi$, and are at a level of the form $\omega \cdot 2 \cdot \alpha$ for some $\alpha < \omega_1$.  Let $T_\tau$ denote those nodes in the game tree of $\mathcal{G}^{\text{Filt}(\mathcal{C})}_{\omega_1}(M)$ that are according to $\tau$.  Then the map
%%\[
%%t \mapsto p(t)
%%\]
%%is an order-preserving map from $T_\psi \to T_\phi$.  The tree $T_\phi$ has no uncountable branches, since $\tau$ is a winning strategy for Player 1 in $\mathcal{G}^{\text{Filt}(\mathcal{C})}_{\omega_1}(M)$, and hence $T_\psi$ has no uncountable branches.  So $\psi$ is a winning strategy for \emph{Spoiler} in $\text{EF}_{\omega_1}(M,P)$.

\end{proof}

\begin{lemma}\label{lem_Equiv_Various_P2WS}
Let $M$ be an $R$-module.  The following are equivalent:
\begin{enumerate}
 \item\label{item_WS2_FreeFilt} Player 2 has a winning strategy in the Free Filtration Game of length $\omega_1$ on $M$ (i.e., the game $\mathcal{G}^{\text{Filt}(\mathcal{C})}_{\omega_1}(M)$ where $\mathcal{C} = \{ R \}$);
 \item\label{item_WS2_EF} \emph{Duplicator} has a winning strategy in $\text{EF}_{\omega_1}(M,F)$, where $F$ the free $R$-module on $\omega_1$ generators;
 \item\label{item_PotentFree} $M$ is $\sigma$-closed potentially free.
\end{enumerate}

\noindent If $R$ is a ring and ``projective = free" for $R$-modules in all $\sigma$-closed forcing extensions---e.g., if $R = \mathbb{Z}$---then these are also equivalent to:
\begin{enumerate}[resume]
 \item\label{item_WS2_ProjFilt} Player 2 has a winning strategy in the Projective Filtration Game of length $\omega_1$ on $M$;

 \item\label{item_PotentProj} $M$ is $\sigma$-closed potentially projective.
\end{enumerate}
\end{lemma}
\begin{proof}
The equivalence of \eqref{item_WS2_FreeFilt} with \eqref{item_PotentFree} is just Lemma \ref{lem_equiv_WS2Filter_Potent}.  The equivalence of \eqref{item_WS2_EF} with \eqref{item_PotentFree} is well-known, see e.g.\ Nadel-Stavi~\cite{MR462942}.  If freeness is equivalent to projectivity for $R$-modules in all $\sigma$-closed forcing extensions, then the equivalence of \eqref{item_PotentFree} with \eqref{item_PotentProj} is trivial.  Parts \eqref{item_PotentProj} and \eqref{item_WS2_ProjFilt} are equivalent by Lemma \ref{lem_equiv_WS2Filter_Potent}.
\end{proof}

Theorem \ref{thm_P1WSFG_implies_SplrWS} and Lemma \ref{lem_Equiv_Various_P2WS} yield:
\begin{corollary}\label{cor_DetImpliesDet}
Let $G$ be a $\mathbb{Z}$-module; i.e., an abelian group.  Determinacy of the Free Filtration Game of length $\omega_1$ on $G$, or determinacy of the Projective Filtration Game of length $\omega_1$ on $G$, implies determinacy of $\text{EF}_{\omega_1}(G,F)$ for all free abelian groups $F$.
\end{corollary}

\section{Determinacy of filtration games and other consequences of Martin's Maximum}\label{sec_MainThms}

The key consequence of Martin's Maximum that we will use is the following principle introduced by Fuchino-Usuba (but under a different name):
\begin{definition}\label{def_RP_internal}
$\text{RP}_{\text{internal}}$ asserts that for all uncountable $X$, all stationary $S \subseteq [X]^\omega$, all regular $\theta$ such that $X \in H_\theta$, and all first order structures $\mathfrak{A}$ in a countable language extending $(H_\theta,\in, X, S)$, there exists a $W$ such that:
\begin{enumerate}
 \item $|W|=\omega_1 \subset W$;
 \item $W \prec \mathfrak{A}$; and
 \item $S \cap W \cap [W \cap X]^\omega$ is stationary in $[W \cap X]^\omega$. (The fact that we require $S \cap W \cap [W \cap X]^\omega$ to be stationary, rather than just $S \cap [W \cap X]^\omega$, is why we call this kind of reflection ``internal").
\end{enumerate}
\end{definition}

The principle $\text{RP}_{\text{internal}}$ can also be characterized by a strong form of Chang's Conjecture (Fuchino-Usuba~\cite{FuchinoUsuba}).  It is also equivalent to a certain L\"owenheim-Skolem-Tarski property of Stationary Logic; roughly speaking, $\text{RP}_{\text{internal}}$ is equivalent to the downward reflection of statements of the form
\[
\text{stat} Z  \ \text{aa} U_1 \dots \text{aa} U_k \ \phi(Z,U_1,\dots,U_k) 
\]
to a substructure of size $<\aleph_2$; here \emph{aa} is the ``almost all" quantifier and \emph{stat} is the ``stationarily many" quantifier of Stationary Logic, introduced by Shelah~\cite{MR376334}.  We will not prove this equivalence, but point out that it is very similar to the proof of the main theorem of Cox~\cite{Cox_Pi11} (which was, in turn, modeled after Fuchino et al.~\cite{FuchinoEtAl_DRP_LST}).

%We need an easy fact:
%\begin{lemma}\label{lem_W_IS}
%Suppose $X$, $S$, and $W$ are as in Definition \ref{def_RP_internal}.  Then $W$ is ``internally stationary at $[X]^\omega$", i.e.
%\[
%W \cap [W \cap X]^\omega \text{ is stationary in } [W \cap X]^\omega.
%\]  
%\end{lemma}
%\begin{proof}
%This is simply because $W \cap [W \cap X]^\omega$ contains the stationary set $S \cap W \cap [W \cap X]^\omega$.  
%\end{proof}

\subsection{$\aleph_2$-compactness of potentially filtered modules}

Given a class $\Gamma$ of $R$-modules, and an $R$-module $M$, let
\[
\wp^\Gamma_{\omega_2}(M):=\{ W \in \wp_{\omega_2}(M) \ :  \  \langle W \rangle^M_R \in \Gamma  \}
\]
\noindent We will say that:
\begin{itemize}
 \item  \textbf{$\boldsymbol{\Gamma}$ is $\boldsymbol{\omega_2}$-club-compact} if for every $R$-module $M$: if $\wp^\Gamma_{\omega_2}(M)$ contains a closed unbounded subset of $\wp_{\omega_2}(M)$, then $M \in \Gamma$.

 \item \textbf{$\boldsymbol{\Gamma}$ is $\boldsymbol{\omega_2}$-compact} if for every module $M$:  if $\wp^\Gamma_{\omega_2}(M) = \wp_{\omega_2}(M)$, then $M \in \Gamma$.
\end{itemize}

\noindent Club compactness is formally stronger than compactness, but in many natural situations they are equivalent.\footnote{E.g., if $\Gamma$ is closed under submodules, then $\wp^\Gamma_{\omega_2}(M)$ contains a club in $\wp_{\omega_2}(M)$ if and only if $\wp^\Gamma_{\omega_2}(M) = \wp_{\omega_2}(M)$.}

The following is the key use of $\text{RP}_{\text{internal}}$:
\begin{theorem}\label{thm_aleph_2_compact}
Assume $\text{RP}_{\text{internal}}$.  Let $R$ be a ring of size at most $\aleph_1$ and $\mathcal{C}$ be a collection of countably-presented $R$-modules.  Then the class of $\sigma$-closed potentially $\mathcal{C}$-filtered modules is $\aleph_2$-club-compact. 
\end{theorem}
\begin{proof}
Suppose $M$ is an $R$-module and club-many elements of $\wp_{\omega_2}(M)$ generate $\sigma$-closed potentially $\mathcal{C}$-filtered submodules of $M$; let $D$ denote this club.  Suppose toward a contradiction that $M$ is not $\sigma$-closed potentially $\mathcal{C}$-filtered.  Then by the equivalence of parts \eqref{item_M_SigmaPotCFiltered} and \eqref{item_aa_stat_NoSizeRest} of Corollary \ref{cor_StatLogicChar_SigmaPotCfilter}, the set
\begin{equation*}
\begin{split}
\Big\{ Z \in [M]^\omega \ : \ \text{For stationarily many } Z' \in [M]^\omega,  \ \langle Z' \rangle/ \langle Z \rangle \text{ is } \mathcal{C} \text{-filtered}   \Big\}
\end{split}
\end{equation*}
\noindent does \textbf{not} contain a club in $[M]^\omega$; let $S$ denote its complement.  Then $S$ is stationary, and for every $Z \in S$,
\[
C_Z:= \{ Z' \in [M]^\omega \ : \   Z' \supseteq Z \text{ and } \langle Z' \rangle/\langle Z \rangle \text{ is } \textbf{not } \mathcal{C} \text{-filtered} \} 
\]
contains a club in $[M]^\omega$.

Fix a large regular $\theta$.  By $\text{RP}_{\text{internal}}$ there exists a
\[
W \prec \Big( H_\theta,\in,D,S,\langle C_Z \ : \ Z \in S \rangle \Big)
\]  
such that $|W|=\omega_1 \subset W$  and
\begin{equation}\label{eq_S_cap_W_stat}
S \cap W \cap [W \cap M]^\omega \text{ is stationary in } [W \cap M]^\omega.
\end{equation}

Note that since $R \in W$ and $|R| \le \omega_1 \subset W$, $W \cap M$ is already closed under scalar multiplication from $R$, and hence $W \cap M = \langle W \cap M \rangle^M$; i.e.\ $W \cap M$ is an $R$-submodule of $M$.  Since $D$ contains a club in $\wp_{\omega_2}(M)$ and $D \in W$, then by Lemma \ref{lem_CharClubElemSub}, $W \cap M \in D$.  Hence,
\begin{equation}\label{eq_W_C_filtered}
W \cap M \text{ is } \sigma \text{-closed potentially } \mathcal{C} \text{-filtered.}
\end{equation}

\noindent Since $W \cap M$ is an $R$-module of size $\omega_1$, Corollary \ref{cor_Aleph_1_gen_CFilt} and \eqref{eq_W_C_filtered} ensure that there is a $\mathcal{C}$-filtration
\[
\Big\langle \langle Z_i \rangle \ : \ i < \omega_1 \Big\rangle
\]
of $W \cap M$, where each $Z_i$ is countable.  By Lemma \ref{lem_ConcatenateFilt}, it follows that 
\begin{equation}\label{eq_AllPairs}
\forall k < \omega_1   \ \forall \ell < \omega_1 \ \ k < \ell \ \implies  \langle Z_\ell \rangle / \langle Z_k \rangle \text{ is } \mathcal{C} \text{-filtered}. 
\end{equation}

Now $\{ Z_i \ : \ i < \omega_1 \}$ is closed unbounded in $[W \cap M]^\omega$, so by \eqref{eq_S_cap_W_stat} there is some $i < \omega_1$ such that $Z_i \in S \cap W$. It follows that $C_{Z_i}$ is an element of $W$, and hence, by Lemma \ref{cor_ClubRef}, that
\[
C_{Z_i} \cap [W \cap M]^\omega \text{ contains a club in } [W \cap M]^\omega.
\]
Then there is some $j > i$ such that $Z_j \in C_{Z_i}$.  Then by definition of $\mathcal{C}_{Z_i}$, $\langle Z_j \rangle/ \langle Z_i \rangle$ is not $\mathcal{C}$-filtered; this contradicts \eqref{eq_AllPairs}.

\end{proof}

Theorem \ref{thm_aleph_2_compact} and Corollary \ref{cor_SameClasses} yield:

\begin{corollary}
Assume $\text{RP}_{\text{internal}}$.  Then for every ring $R$ of size at most $\aleph_1$, the class of $\sigma$-closed potentially projective modules is $\aleph_2$-club-compact.
\end{corollary}

\subsection{Proof of Theorem \ref{thm_Cox_MM_Det}}

Theorem \ref{thm_Cox_MM_Det} from the introduction follows from Theorem \ref{thm_aleph_2_compact} and the following theorem:

\begin{theorem}\label{thm_QH_compact_Det}
Assume $R$ is a ring and $\mathcal{C}$ is a quotient-hereditary collection of countably presented $R$-modules.  Suppose the class of $\sigma$-closed potentially $\mathcal{C}$-filtered modules is $\aleph_2$-club-compact.  Then for every $R$-module $M$, the game $\mathcal{G}^{\text{Filt}(\mathcal{C})}_{\omega_1}(M)$ is determined.
\end{theorem}
\begin{proof}
Suppose Player 2 does not have a winning strategy in $\mathcal{G}^{\text{Filt}(\mathcal{C})}_{\omega_1}(M)$.  By Lemma \ref{lem_equiv_WS2Filter_Potent}, $M$ is not $\sigma$-closed potentially $\mathcal{C}$-filtered.  By the assumed $\aleph_2$-club compactness, there is a $W \in \wp_{\omega_2}(M)$ such that $\langle W \rangle^M_R$ is not $\sigma$-closed potentially $\mathcal{C}$-filtered.

Note that in the game $\mathcal{G}^{\text{Filt}(\mathcal{C})}_{\omega_1}(\langle W \rangle)$, Player 1 has an easy winning strategy:  he simply enumerates $W$.  Such a game cannot last $\omega_1$ rounds, because if it did, it would yield a $\mathcal{C}$-filtration of $\langle W \rangle$.  Then by Lemma \ref{lem_UpClosed} and the assumption that $\mathcal{C}$ is quotient-hereditary, Player 1 has a winning strategy in the game $\mathcal{G}^{\text{Filt}(\mathcal{C})}_{\omega_1}(M)$.

\end{proof}

\begin{corollary}
$\text{RP}_{\text{internal}}$ implies that for every hereditary ring $R$ of size at most $\aleph_1$, Projective Filtration games of length $\omega_1$ for $R$-modules are determined.

The same conclusion holds for Free Filtration Games of length $\omega_1$.
\end{corollary}

\subsection{Proof of Theorem \ref{thm_Cox_ClosedDL}}

A partial order $(I,\le)$ is called $\boldsymbol{<\mu}$\textbf{-directed} (for a regular cardinal $\mu$) if every $<\mu$-sized subset of $I$ has an upper bound.  A direct system is $<\mu$-directed if its underlying index set is $<\mu$-directed.

\begin{theorem}\label{thm_CompactImpliesClosedDL}
Suppose $R$ is a ring of size at most $\aleph_1$ and $\Gamma$ is a collection of $R$-modules such that:
\begin{enumerate}
 \item $\Gamma$ is $\aleph_2$-club-compact;
 \item $\Gamma$ is downward closed under pure submodules.
\end{enumerate}
Then $\Gamma$ is closed under $<\aleph_2$-directed limits of $R$-module homomorphisms.
\end{theorem}

\begin{proof}
Suppose $\mathcal{D}$ is a directed system of structures from $\Gamma$, indexed by the $< \aleph_2$-directed partial order $(I,\le)$.  For $i \le j$ in $I$ let $\pi_{i,j}: M_i \to M_j$ be the associated homomorphism.  Note that we do \emph{not} assume the $\pi_{i,j}$'s are injective.

Let $M_{\mathcal{D}}$ denote the direct limit of $\mathcal{D}$, and suppose toward a contradiction that $M_{\mathcal{D}} \notin \Gamma$.  By $\aleph_2$-club-compactness of $\Gamma$, there is a stationary $T \subset \wp_{\omega_2}(M_{\mathcal{D}})$ consisting of submodules of $M_{\mathcal{D}}$ that fail to be in $\Gamma$.  By Lemma \ref{lem_CharClubElemSub} there is a
\[
W \prec (H_\theta,\in,T,\mathcal{D})
\]
such that $|W|=\omega_1 \subset W$ and $W \cap M_{\mathcal{D}} \in T$; so
\begin{equation}\label{eq_W_Minfty_NotFilt}
W \cap M_{\mathcal{D}} \notin \Gamma.
\end{equation}
 
Since $|W \cap I| \le |W| = \aleph_1$, the $<\aleph_2$-directedness of $I$ ensures there is an $i_W \in I$ above all indices in $W \cap I$.  By assumption on $\mathcal{D}$, $M_{i_W} \in \Gamma$.  Since $\Gamma$ is downward closed under pure embeddings, to obtain a contradiction it will suffice to find a pure embedding from $W \cap M_{\mathcal{D}}$ into $M_{i_W}$.

For $i \in I$ and $x \in M_i$, $[i,x]_{\mathcal{D}}$ denotes the equivalence class of the thread $\{ \pi_{i,j}(x) \ : \ j \ge_I i \}$ in the direct limit $M_{\mathcal{D}}$.  Note that, by elementarity of $W$, every element of $W \cap M_{\mathcal{D}}$ is of the form $[i,x]_{\mathcal{D}}$ for some $i \in W \cap I$ and some $x \in W \cap M_i$.  Define
\[
\rho: W \cap M_{\mathcal{D}} \to M_{i_W}
\]
by 
\[
[i,x]_{\mathcal{D}} \mapsto  \pi_{i,i_W}(x),
\]
where $i$ is any element of $W \cap I$ and $x$ is any element of $W \cap M_i$; note that under those assumptions, $\pi_{i,i_W}$ exists (because $i_W$ is above every index in $W$).  The map $\rho$ is well-defined, because if $[i,x]=[j,y]$ where $i,j,x,y \in W$, then by elementarity of $W$ there is some $k \in W \cap I$ with $k \ge i,j$ such that $\pi_{i,k}(x) = \pi_{j,k}(y) =: z$.  Since $i_W \ge k$, it follows from commutativity of the system $\mathcal{D}$ that that $\pi_{i,i_W}(x) =  \pi_{k,i_W}(z)= \pi_{j,i_W}(y)$.  That $\rho$ is injective is trivial, based on the definition of the $\mathcal{D}$-equivalence relation.  

To see that $\rho$ is an $R$-module homomorphism, first observe that $|R| \le \omega_1 \subset W$ and $R \in W$, and it follows (by elementarity of $W$) that 
\begin{equation}\label{eq_R_subset_W}
R \subset W.
\end{equation}
Now suppose $r,s \in R$, $i,j \in W \cap I$, $x \in W \cap M_i$, and $y \in W \cap M_j$.  By \eqref{eq_R_subset_W}, both $r$ and $s$ are in $W$; since $x$, $y$, $i$, $j$, and $k$ are also in $W$, we have
\begin{equation}\label{eq_LC_in_W}
r\pi_{i,k}(x) + s\pi_{j,k}(y) \text{ is an element of } W \cap M_k.
\end{equation}
Then
\begin{equation*}
\begin{split}
\rho \Big( r[i,x] + s[j,y] \Big)= \rho  \Big(  [k, r \pi_{i,k}(x) + s \pi_{j,k}(y)] \Big) = \pi_{k,i_W} \Big(  r \pi_{i,k}(x) + s \pi_{j,k}(y) \Big) \\
= r \pi_{i,i_W}(x) + s\pi_{j,i_W}(y) = r \rho\big( [i,x] \big) + s \rho \big( [j,y] \big),
\end{split}
\end{equation*}
where the second equality is by \eqref{eq_LC_in_W} and the definition of $\rho$.

To see that $\rho$ is a pure embedding, suppose 
\[
Av = b
\]
where:
\begin{itemize}
 \item $A$ is a (finite) matrix with entries from $R$;
 \item $b = \Big( \pi_{i_1, i_W}(b_1), \dots, \pi_{i_r, i_W}(b_r) \Big)^t$ is a vector of elements of the range of $\rho$, where each $i_k$ and $b_k$ come from $W$.
 \item $v$ is a vector from $M_{i_W}$.
\end{itemize}
Then the index $i_W$ and the vector $v$ witness that
\begin{equation}\label{eq_WhatHthetaBelieves}
\begin{split}
(H_\theta,\in) \models & \ \exists i \ \ i \ge i_1, \dots, i_r \ \text{ and } \\
 & Ax = \Big( \pi_{i_1,i}(b_1) , \dots, \pi_{i_r, i}(b_r) \Big)^t \text{ is solvable in } M_i.
\end{split}
\end{equation}
Now \eqref{eq_R_subset_W} (and the fact that $A$ is a finite matrix) ensures that $A \in W$.  The other parameters from \eqref{eq_WhatHthetaBelieves} are in $W$ too, and hence the $i$ from \eqref{eq_WhatHthetaBelieves} can be taken to be in $W \cap I$.  Furthermore, the solution in $M_i$ can be taken to come from $W \cap M_i$, for the same reason.  Say $\Big( u_1, \dots, u_\ell \Big)$ is a vector in $W \cap M_i$ such that 
\[
A\Big(u_1, \dots, u_\ell \Big)^t = \Big( \pi_{i_1,i}(b_1) , \dots, \pi_{i_r, i}(b_r) \Big)^t.
\]
Then $\Big(  \pi_{i,i_W}(u_1), \dots, \pi_{i,i_W}(u_\ell) \Big)$ is a vector from the range of $\rho$, and
\[
A \Big( \pi_{i,i_W}(u_1), \dots, \pi_{i,i_W}(u_\ell) \Big)^t = \Big( \pi_{i_1,i_W}(b_1), \dots, \pi_{i_r,i_W}(b_r)  \Big)^t = b.
\]

\end{proof}

We can now finish the proof of Theorem \ref{thm_Cox_ClosedDL}.   Assume $\text{RP}_{\text{internal}}$, $R$ is a ring of size at most $\aleph_1$, and $\mathcal{C}$ is a quotient-hereditary collection of countably-presented $R$-modules.  Let $\Gamma$ denote the class of $\sigma$-closed potentially $\mathcal{C}$-filtered modules, which by Corollary \ref{lem_equiv_WS2Filter_Potent} is the same as the class of modules $M$ such that Player 2 has a winning strategy in the $\mathcal{C}$-filtration game of length $\omega_1$ on $M$.  By Theorem \ref{thm_Cox_MM_Det}, such games are determined; this determinacy, together with Lemma \ref{lem_UpClosed}, implies that $\Gamma$ is downward closed under submodules (so, in particular, under pure submodules).  And Theorem \ref{thm_aleph_2_compact} guarantees that $\Gamma$ is $\aleph_2$-club-compact.  So by Theorem \ref{thm_CompactImpliesClosedDL}, $\Gamma$ is closed under $<\aleph_2$-directed limits, which completes the proof of Theorem \ref{thm_Cox_ClosedDL}.

\begin{remark}
Although we will not need the following, we note that Theorem \ref{thm_CompactImpliesClosedDL} generalizes easily to other model-theoretic settings.  For example, suppose $\mathcal{L}$ is a first order language, $\Gamma$ a collection of $\mathcal{L}$-structures that is $\aleph_2$-club compact,  and there exists a $<\aleph_2$-sized collection $\Sigma$ of $\mathcal{L}$-formulas such that $\Gamma$ is downward closed under injective $\Sigma$-preserving maps.  Then whenever $\mathcal{D}$ is a $<\aleph_2$-directed system of $\Sigma$-preserving (not necessarily injective) maps between elements of $\Gamma$, the direct limit of $\mathcal{D}$ is in $\Gamma$.
\end{remark}

\begin{corollary}\label{cor_RP_HeredAllSigmaClosed}
Assume $\text{RP}_{\text{internal}}$ and that $R$ is a ring of size at most $\aleph_1$ that is hereditary in all $\sigma$-closed forcing extensions.  Then the class of $\sigma$-closed potentially projective $R$-modules is closed under $<\aleph_2$-directed limits.
\end{corollary}
\begin{proof}
If $R$ is hereditary in all $\sigma$-closed forcing extensions, then the class of $\sigma$-closed potentially projective modules is closed under submodules (in particular, under pure submodules).  The result then follows from Corollary \ref{cor_SameClasses}, Theorem \ref{thm_aleph_2_compact}, and Theorem \ref{thm_CompactImpliesClosedDL}.
\end{proof}

\begin{lemma}\label{lem_CtbleHered}
If $R$ is a countable, hereditary ring, then $R$ remains hereditary in all $\sigma$-closed forcing extensions.
\end{lemma}
\begin{proof}
Suppose $R$ is countable and hereditary, and let $V[G]$ be a $\sigma$-distributive forcing extension of $V$. By Chapter 2E of \cite{MR1653294}, a ring $R$ is hereditary if and only if all ideals in $R$ are projective (as $R$-modules); so it suffices to show that in $V[G]$, all ideals in $R$ are projective.  Suppose, in $V[G]$, that $I$ is an ideal of $R$.  Since $R$ was countable in $V$ and $V[G]$ added no new countable sets, $I$ was already an element of $V$.  Since $V \models$ ``$R$ is hereditary", then $I$ is projective in $V$, and this is easily upward absolute to $V[G]$. 
\end{proof}

\begin{corollary}
Assume $\text{RP}_{\text{internal}}$ and that $R$ is a \textbf{countable}, hereditary ring.  Then the class of $\sigma$-closed potentially projective $R$-modules is closed under $<\aleph_2$-directed limits.
\end{corollary}
\begin{proof}
This follows immediately from Lemma \ref{lem_CtbleHered} and Corollary \ref{cor_RP_HeredAllSigmaClosed}.
\end{proof}

\subsection{Proof of Theorem \ref{thm_AnswerBaldwin}}\label{subsec_AnswerBaldwin}

%
%If $\mathcal{C}$ is a collection of countably presented modules, recall from Section \ref{} that  
%\[
%\Gamma  \big( \mathcal{G}^{\text{Filt}(\mathcal{C})}_{\omega_1} , \text{P1} \big)
%\]
%denotes the class of modules $M$ such that Player 1 has a winning strategy in the game $\mathcal{G}^{\text{Filt}(\mathcal{C})}_{\omega_1}(M)$.  Let 
%\[
%\Gamma  \big( \mathcal{G}^{\text{Filt}(\mathcal{C})}_{\omega_1} , \boldsymbol{\neg}\text{P2} \big)
%\]
%denote the class of modules $M$ such that Player 2 does \textbf{not} have a winning strategy in $\mathcal{G}^{\text{Filt}(\mathcal{C})}_{\omega_1}(M)$.  By Lemma \ref{}, $\Gamma  \big( \mathcal{G}^{\text{Filt}(\mathcal{C})}_{\omega_1} , \boldsymbol{\neg}\text{P2} \big)
%$ is the same as the class of modules that are not $\sigma$-closed potentially $\mathcal{C}$-filtered.  
%

Consider the statement:
\begin{quote}
$\Phi :\equiv$ ``The class of $\mathbb{Z}$-modules (i.e., abelian groups) $G$ that are \textbf{not} $\sigma$-closed potentially free, under the pure embeddability ordering, is an AEC."
\end{quote}

\noindent We will show that $\Phi$ is independent of ZFC.  One direction has basically already been done:
\begin{theorem}
$\text{RP}_{\text{internal}}$ implies that for all rings $R$ of size at most $\aleph_1$ and all quotient-hereditary classes $\mathcal{C}$ of countably presented $R$-modules, the class of $R$-modules that are \textbf{not} $\sigma$-closed potentially $\mathcal{C}$-filtered is (under the pure embeddability order) an AEC. Moreover, its L\"owenheim-Skolem number is at most $\aleph_1$.

In particular, $\Phi$ holds and the L\"owenheim-Skolem number is $\aleph_1$ (this is the special case where $R = \mathbb{Z}$ and $\mathcal{C} = \{ R \}$).  
\end{theorem}
\begin{proof}
By Lemma \ref{lem_equiv_WS2Filter_Potent}, the class in the statement of the theorem is the same as the class of modules for which Player 2 has no winning strategy in the $\mathcal{C}$-filtration game of length $\omega_1$.  By Theorems \ref{thm_aleph_2_compact} and  \ref{thm_QH_compact_Det}, this game is determined, and hence the class from the previous sentence is the same as the class of modules for which Player 1 has a winning strategy in that game.  And since $\mathcal{C}$ is quotient-hereditary, Theorem \ref{thm_AlwaysAEC} ensures this class, under the pure embeddability order, is always an AEC.  Ordinarily, $2^{\aleph_0}$ is the best-known upper bound for the L\"owenheim-Skolem number, but Theorem \ref{thm_aleph_2_compact} ensures that every non-$\sigma$-closed potentially $\mathcal{C}$-filtered module has another such submodule of size $\aleph_1$.  So the L\"owenheim-Skolem number of the AEC is $\aleph_1$.
\end{proof}

%Note that, by Lemma \ref{}, the class in the statement is the same as
%\[
%\Gamma  \big( \mathcal{G}^{\text{Filt}(\mathcal{C})}_{\omega_1} , \boldsymbol{\neg}\text{P2} \big)
%\]
%where $\mathcal{C} = \{ \mathbb{Z} \}$ (and the game is for $\mathbb{Z}$-modules).  So, if $\text{RP}_{\text{internal}}$ holds, Corollary \ref{} ensures that the Free Filtration Game of length $\omega_1$ is always determined, and hence 
%\[
%\Gamma  \big( \mathcal{G}^{\text{Filt}(\mathcal{C})}_{\omega_1} , \text{P1} \big) = \ \Gamma  \big( \mathcal{G}^{\text{Filt}(\mathcal{C})}_{\omega_1} , \boldsymbol{\neg}\text{P2} \big).
%\]
%
%\noindent Since $\Gamma  \big( \mathcal{G}^{\text{Filt}(\mathcal{C})}_{\omega_1} , \text{P1} \big)$ (under the pure embeddability relation) is always an AEC by Corollary \ref{cor_AlwaysAEC}, this completes the proof of the consistency of $\Phi$ (relative to consistency of Martin's Maximum, which is in turn consistent if supercompact cardinals are).
%
%Note actually that in the current situation, we know more about the LS number of $\Gamma  \big( \mathcal{G}^{\text{Filt}(\mathcal{C})}_{\omega_1} , \text{P1} \big) $ than usual.  In general we can get $2^{\aleph_0}$-sized substructures, but if $\text{RP}_{\text{internal}}$ holds we can get $\omega_1$-sized substructures.
%
%
%
%
%
%

To show that $\neg \Phi$ is relatively consistent with ZFC, we will use:

\begin{theorem}[Mekler et al.~\cite{MR1191613}]\label{thm_Square}
Suppose $\kappa$ is an infinite cardinal and Jensen's $\square_\kappa$ principle holds.  Then there exists an abelian group $G$ of cardinality $\kappa^+$ that is $\kappa^+$-free (i.e.\ every $< \kappa^+$-sized subgroup of $G$ is free), but the game 
\[
\text{EF}_{\omega_1}\big( F_{\omega_1},G\big)
\]
is not determined, where $F_{\omega_1}$ denotes the free abelian group of size $\omega_1$.

\end{theorem}

\begin{remark}
The proof of Theorem \ref{thm_Square} from \cite{MR1191613} only dealt with the case $\kappa = \omega_1$, but the proof goes through for any $\kappa$.
\end{remark}

\begin{theorem}\label{thm_CON_NegDB2_Not_AEC}
Suppose that Jensen's $\square_\kappa$ principle holds for class-many $\kappa$ (failure of this requires consistency of large cardinals).  Then $\Phi$ fails.
\end{theorem}
\begin{proof}
For brevity, let $\Gamma_{\neg \text{P2}}$ be the class of abelian groups such that Player 2 has no winning strategy in the Free Filtration Game of length $\omega_1$ on $G$.  We show that the L\"owenheim-Skolem requirement \ref{item_LS_axiom} from Definition \ref{def_AEC} fails for this class (under the pure subgroup order).

Suppose toward a contradiction that $\Gamma_{\neg \text{P2}}$ did have an L\"owenheim-Skolem number, say, $\mu$.  By assumption, there is a $\kappa$ larger than $\mu$ such that Jensen's $\square_\kappa$ holds.  By Theorem \ref{thm_Square}, there is an almost free abelian group $G$ of cardinality $\kappa^+$ such that $\text{EF}_{\omega_1}\big( G, F_{\omega_1} \big)$ is not determined.  Then by Corollary \ref{cor_DetImpliesDet}, the Free Basis game of length $\omega_1$ on $G$ is not determined; so, in particular, Player 2 does not have a winning strategy in the Free Basis game of length $\omega_1$ on $G$.  Since $G$ is almost free and $|G|=\kappa^+$, however, every $\le \kappa$-sized subgroup of $G$ is free.  In particular, every $\mu$-sized subgroup of $G$ is free, and hence (trivially) $\sigma$-closed potentially free.  In other words, $G \in \Gamma_{\neg \text{P2}}$ but no $\mu$-sized subgroup of $G$ is in $\Gamma_{\neg \text{P2}}$.  This contradicts our assumption that $\mu$ is the L\"owenheim-Skolem number for the class $\Gamma$.
\end{proof}

\section{Concluding remarks}\label{sec_Conclusion}

The \textbf{Weak Reflection Principle (WRP)} is defined just like $\text{RP}_{\text{internal}}$ (Definition \ref{def_RP_internal}), \emph{except} that one only requires $S \cap [W \cap X]^\omega$ to be stationary, rather than requiring $S \boldsymbol{\cap W} \cap [W \cap X]^\omega$ to be stationary.  Clearly $\text{RP}_{\text{internal}}$ implies WRP, but it is not known if the converse holds.  Note that the ``internality" part of $\text{RP}_{\text{internal}}$ was used at crucial point near the end of the proof of Theorem \ref{thm_aleph_2_compact}:  we used that $Z_i \in W$ to ensure that $C_{Z_i} \in W$, which was needed to conclude that $C_{Z_i} \cap [W \cap M]^\omega$ contained a club.
% \item In the proof of Theorem \ref{thm_ImproveFarahMag}, the internality to $W$ was used in the verification of \eqref{eq_WcapXR}.
%\end{itemize}
Furthermore, $\text{RP}_{\text{internal}}$ can be shown to be equivalent to a kind of downward L\"owenheim-Skolem-Tarki property for certain formulas in stationary logic, and the internality is used at crucial places in the equivalence (via an argument similar to the  main theorem of Cox~\cite{Cox_Pi11}).

These observations raise, yet again, an open problem that has appeared in similar forms in other places (e.g.\ Beaudoin~\cite{MR877870}, Krueger~\cite{MR2674000}, and Cox~\cite{Cox_RP_IS}).  The principles $\text{RP}_{\text{IS}}$ and $\text{RP}_{\text{IU}}$ are further variants whose definition we will not provide (see Krueger~\cite{MR2674000}); the principle $\text{RP}_{\text{IU}}$ is equivalent to \textbf{Fleissner's Axiom R}, by Fuchino-Usuba~\cite{FuchinoUsuba}.  The following implications are known:
\begin{equation}\label{eq_ChainRPs}
\text{RP}_{\text{internal}} \ \implies \ \text{RP}_{\text{IS}} \ \implies \ \text{RP}_{\text{IU}} \ \implies \text{WRP}.
\end{equation}

\begin{question}
Which pairs of principles, if any, from the chain \eqref{eq_ChainRPs} are equivalent?
\end{question}
\noindent It is currently not known if \emph{any} pair from \eqref{eq_ChainRPs} is equivalent.  An obviously related question is:
\begin{question}
Do any of the consequences of $\text{RP}_{\text{internal}}$ from this paper follow from WRP?
\end{question}

\end{document}